\documentclass[11pt,a4paper]{amsart}
\usepackage{latexsym}
\usepackage{amsmath, amsthm}
\usepackage{graphicx} 
\usepackage{amsfonts,tikz}
\usepackage{hyperref}
\usepackage{amssymb}
\usetikzlibrary{arrows}

\usetikzlibrary{decorations.pathreplacing}

  \numberwithin{equation}{section}

\DeclareMathOperator{\inv}{inv}

\DeclareMathOperator{\des}{des}
\DeclareMathOperator{\Des}{Des}
\DeclareMathOperator{\Av}{Av}
\DeclareMathOperator{\einv}{einv}
\DeclareMathOperator{\Inv}{Inv}
\DeclareMathOperator{\cl}{cl}

\theoremstyle{plain}
\newtheorem{thm}{Theorem}[section]
\newtheorem{pro}[thm]{Proposition}
\newtheorem{lem}[thm]{Lemma}
\newtheorem{con}[thm]{Conjecture}
\newtheorem{prob}[thm]{Problem}
\newtheorem{cor}[thm]{Corollary}

\theoremstyle{definition}
\newtheorem{defn}[thm]{Definition}
\newtheorem*{acknow}{Acknowledgments}

\theoremstyle{remark}
\newtheorem{rmk}[thm]{Remark}
\newtheorem{exm}[thm]{Example}

\newcommand{\mc}{{\footnotesize\left[ \begin{array}{c}  m \\
\left\lfloor \frac{|I_{1}|+1}{2}\right\rfloor,\ldots,\left\lfloor \frac{|I_{s}|+1}{2}\right\rfloor \end{array} \right]_{x^2}}}

\newcommand{\mcdes}{\footnotesize{ \left[ \begin{array}{c} \bt+\dt \\
\bf b, \,\bf d \end{array} \right]_{x^2}}}

\newcommand{\mcdesa}{\footnotesize{ \left[ \begin{array}{c} \bt+\dt \\
\bf b, \,\bf d \end{array} \right]_{x^2}}}

\newcommand{\dt}{d}
\newcommand{\bt}{b}
\newcommand{\sm}{\sigma^{-1}}

\newcommand{\N}{\mathbb N}
\newcommand{\PP}{\mathbb P}
\newcommand{\CC}{\mathbb C}
\newcommand{\f}{(-1)^{\ell(\sigma)}x^{L(\sigma)}}

\newcommand{\D}{{\mathcal D}}

\begin{document}
	\title{Odd length: odd diagrams and descent classes}
\author{Francesco Brenti and Angela Carnevale} 
\address{Francesco Brenti,	Dipartimento di Matematica 
					Universit\`{a} di Roma ``Tor Vergata''
					Via della Ricerca Scientifica, 1 
					00133 Roma, Italy }
	\email{brenti@mat.uniroma2.it}
\address{Angela Carnevale,  School of Mathematics, Statistics and Applied Mathematics, National University of Ireland, Galway, Ireland}
	\email{angela.carnevale@nuigalway.ie}

\begin{abstract}
We define and study odd analogues of classical geometric and combinatorial objects 
associated to permutations, namely odd Schubert varieties, odd diagrams, and
odd inversion sets. We show that there is a bijection between odd inversion sets 
of permutations and acyclic orientations of the Tur\'{a}n graph, that the
dimension of the odd Schubert variety associated to a permutation is the odd 
length of the permutation, and give several necessary conditions for a subset of 
$[n] \times [n]$ to be the odd diagram of a permutation. We also study 
the sign-twisted generating function of the odd length over descent classes 
of the symmetric groups.
\end{abstract}

\maketitle

\section{Introduction}
Motivated by questions in enumerative geometry a new statistic on the symmetric groups was 
introduced and studied in \cite{KV}.
 This statistic combines combinatorial and parity
conditions and is now known as the odd inversion number 
 or odd length (see, e.g., \cite{BC}). 
 It was conjectured in \cite[Conjecture C]{KV} that the
sign-twisted generating function of this new statistic on any quotient of any symmetric group 
is given by an explicit product formula. This conjecture was 
proved in~\cite{BC}. An odd length statistic has also been defined and studied on the hyperoctahedral groups by Stasinski and Voll in \cite{SV1} and \cite{SV2}, 
on the even hyperoctahedral groups by the authors in~\cite{BC2}, and on 
all Weyl groups in~\cite{BC4} and~\cite{Stem}.

Our purpose in this paper is to carry out a further study of this statistic on the symmetric 
groups from the combinatorial, enumerative and geometric point of view. More precisely, we show that, given any permutation $\sigma$,
there is a complex projective variety $X_o(\sigma)$ (which we call an 
{\em odd Schubert variety}) whose dimension is the 
odd length of $\sigma$. We also define and study the odd analogues of two other familiar combinatorial
objects associated to a permutation, namely  diagrams and inversion sets.
We show that there is a simple transformation connecting odd inversion sets and 
odd diagrams, we characterize the subsets of $[n] \times [n]$ that are
odd inversion sets of permutations, and we give several necessary conditions
for a subset of $[n] \times [n]$ to be the odd diagram of a permutation.
Also, we study the sign-twisted generating function of the odd length over 
 descent classes of the symmetric groups. In particular, we give sufficient
 conditions for the generating function to be zero, and compute it explicitly 
 for the alternating permutations and for a family of descent classes which includes
all quotients.

The organization of the paper is as follows. In the next section we
recall definitions and results that we use in the sequel. In \S~\ref{sec:oddiag} 
we introduce and study odd Schubert varieties, odd inversion sets, and
odd diagrams. More precisely, we show that the odd length of a permutation $\sigma$ is the dimension
of its associated odd Schubert variety, and that this variety depends on a subset of the diagram of $\sigma$ 
(which we call the {\em odd diagram} of $\sigma$). Furthermore, that there is a simple transformation
relating the odd diagram with a subset of the inversion set (which we call the
{\em odd inversion set}) of $\sigma$, that there is a bijection between odd 
inversion sets of permutations and acyclic orientations of the Tur\'{a}n graph,
and give several necessary conditions
for a subset of $[n] \times [n]$ to be the odd diagram of a permutation.
In \S~\ref{sec:operations} we
 study the effect that some operations, that can be performed on a descent class, 
 have on the corresponding sign-twisted generating function of the odd length. 
 In \S~\ref{sec:unmixed} we give sufficient conditions on a descent class for its sign-twisted 
 generating function to be zero and we 
compute it explicitly for the
descent classes of the alternating permutations,
and for a general family of descent classes which includes all quotients.
Finally, in \S~\ref{sec:conj}, we present some conjectures and open problems arising
from the present work.

\vspace{1em}

\section{Preliminaries}\label{sec:pre}
In this section we collect some notation and basic facts about symmetric groups that we use in the sequel. 
Besides the combinatorial aspects we recall the geometric facts and definitions about the Schubert variety 
associated with a permutation. 

For $X\subseteq \mathbb N$ we let $X_0$ denote $X\cup\{0\}$. For $m,n \in {\mathbb Z}$, $m\leq n$, 
we let $[m,n]$ denote the set $\{m,m+1,\ldots,n-1,n\}$ and for $n\in \mathbb N$ we let $[n]=[1,n]$. 
Given $J\subseteq [n-1]$ there are unique integers $a_1<\cdots<a_s$ and $b_1<\cdots<b_s$ such 
that $J=[a_1,b_1]\cup\cdots \cup [a_s,b_s]$ and $b_i +1<a_{i+1}$ for $i=1,\dots,s-1$. We call 
the intervals $[a_1,b_1],\ldots,[a_s,b_s]$ the \emph{connected components of $J$}.

For $n \in \N$ we let 
$[n]_q := (1-q^n)/(1-q)$ (so $[0]_q=0 $), and $[n]_q ! := \prod_{i=1}^{n} [i]_q$ (so $[0]_q!= 1$).
For $n_1,\ldots,\,n_k \in \N$ such that $\sum_{i=1}^k n_i =n$ we let
\[
\left[ \begin{array}{c} n \\
n_1,\ldots,n_k \end{array} \right]_{q}:= \frac{[n]_{q}!}{[n_1]_q !\cdot \cdots \cdot [n_k]_q !}.
\]

We refer to \cite{BB} for notation, terminology and basic facts about Coxeter groups. 
The symmetric group $S_n$ is the group of permutations of $[n]$. We let $S=\{s_1,\ldots,s_{n-1}\}$ denote the set of standard generators of $S_n$, where $s_i$ denotes the $i$-th transposition $(i, i+1)$. 
It is well known that $S_n$ is a Coxeter group  with respect to this set of generators and that for $\sigma\in S_n$ the Coxeter length $\ell(\sigma)$ and the descent set $\Des(\sigma)$ have combinatorial interpretations 
\[\ell(\sigma)=|\{(i,j)\in [n]^2 : i<j, \sigma(i)>\sigma(j)\}|\]
and \begin{equation}\label{dessn}
\Des(\sigma)=|\{i\in [n-1] :  \sigma(i)>\sigma(i+1)\}|,
\end{equation}
respectively. In the sequel we often identify the generating set $S$ with the set $[n-1]$.
For $\sigma \in S_n$ the
{\em diagram} of $\sigma$ is
$$
D(\sigma) := \{ (i,j) \in [n]^2 : j<\sigma(i), \sigma^{-1}(j)>i \},
$$
and the {\em inversion set } of $\sigma$ is
$$
\Inv(\sigma) := \{ (i,j) \in [n]^2 : i<j, \, \sigma(i)>\sigma(j) \}.
$$
Note that $|D(\sigma)|= |\Inv(\sigma)| = \ell(\sigma)$, and that 
$(i,\sigma(j)) \in D(\sigma)$ if and only if $(i,j) \in \Inv(\sigma)$, 
for all $(i,j) \in [n]^2 $.

Let $n \in \mathbb N$ and $\{ e_1, \ldots , e_n \}$ be the canonical basis of $\CC^n$.
A {\em flag} in $\CC^n$ is a sequence $(U_1,\ldots , U_{n})$ of subspaces of
$\CC^n$ such that $ U_1 \subset \cdots \subset U_{n}$  and
$\dim(U_i)=i$ for all $i=1, \ldots , n$. The set $F_n$ of all flags in $\CC^n$ is called 
the {\em flag manifold} of $\CC^n$. For $\sigma \in S_n$ we denote by $C(\sigma)$ the
{\em Schubert cell} of $\sigma$. We choose here the definition that is most convenient for our
purposes. Namely, 
$(U_1,\ldots , U_n) \in C(\sigma)$ if and only if there are $(a_{i,j})_{(i,j) \in D(\sigma)} \in \CC^{D(\sigma)}$
such that  
$$
U_k = 
\langle \{e_{\sigma(i)} + \sum_{\{ j : (i,j) \in D(\sigma) \} } a_{i,j} \, e_j  \}_{1 \leq i \leq k} \rangle
$$
for $k \in [n]$. 
It is well known (see, e.g., \cite[(A.4)]{Mac}) and not hard to see, that the  map 
$(a_{i,j})_{(i,j) \in D(\sigma)} \mapsto (U_1,\ldots , U_n)$ is injective. In
particular, $C(\sigma)$ is isomorphic to an affine space of dimension $\ell(\sigma)$.

Recall (see, e.g., \cite[p.\ 209]{GH}) that for $i \in [n]$ the {\em Pl\"{u}cker embedding} $\pi_i$ 
associates to any $i$-dimensional subspace $U$ of $\CC^n$ a point in the projective
space $\PP (\Lambda^i (\CC^n))=\PP^{(^n_{\, i} ) -1}$. More precisely, the image of $U$ under $\pi_i$ is the ${n}\choose{i}$-tuple
$(U_I)_{\{ I \subseteq [n]: |I|=i \}}$ where, for $I \subseteq [n]$, $|I|=i$, $U_I$ is the 
minor whose columns are indexed by the elements in $I$ of a matrix
which has as rows a basis  of $U$. 

One   may thus associate to any flag $(U_1, \ldots , U_{n})$ of $\CC^n$ a point in the
Cartesian product 
$\PP(\Lambda^1 (\CC^n)) \times \PP(\Lambda^2 (\CC^n)) \times \cdots \times \PP(\Lambda^{n-1} (\CC^n))$.
In turn, to any point in this product the {\em Segre embedding} (see, e.g., 
\cite[Chap.\ I, Ex.\ 2.14]{Har}
for the definition and information about the Segre embedding)
associates a point in the projective space $\PP(E)$
where $ E := \CC^n \otimes \Lambda^2(\CC^n) \otimes \cdots \otimes \Lambda^{n-1}(\CC^n)$.
The image of $F_n$ under this composite embedding (which we denote by $\pi$) 
is a complex projective algebraic variety.

The {\em Schubert variety} 
$X(\sigma)$ is the closure of $\pi(C(\sigma))$ in $\pi(F_n)$. It is well known 
that $X(\sigma)$ is a complex projective variety of dimension $\ell(\sigma)$.

\vspace{3mm}

 One of our results concerns generating functions on descent classes 
 of the symmetric groups, which we now define.
For $I,J\subseteq S$, $J\subseteq S \setminus I$ we let
\begin{eqnarray}\label{descl}
\D _J ^I(S_n) &:=&\{\sigma \in S_n \,|\, J\subseteq \Des(\sigma)\subseteq S\setminus I\}, \\ 
\label{quodes}
S_n^I  &:=& \D_{\emptyset} ^{I}(S_n).
\end{eqnarray}
Similarly, for   subsets $X\subseteq S_n$, we denote $\D _J ^I (X):=X\cap \D_J ^I (S_n)$.

To state the main result of~\cite{BC},  which is  a special   case  of
one of our main  results, we need the following definitions. 
Let $n\in \N$. Set:
\begin{align*}
C_{n,\,+} &:=  \{\sigma\in S_n : i+ \sigma(i)\equiv 0 \pmod{2},\;  i=1,\ldots,\,n \}, \\
C_{n,\,-} &:=  \{\sigma\in S_n : i+ \sigma(i)\equiv 1 \pmod{2},\;  i=1,\ldots,\,n \}, \\
C_n       &:=  C_{n,\,+} \cup C_{n,\,-}.
\end{align*}
Note that
\[
C_n = \{ \sigma \in S_n  : i \equiv j \pmod{2} \Rightarrow  \sigma(i) \equiv \sigma(j) \pmod{2}, \mbox{ for all } 
i,j \in [n] \}.
\]
Elements in $C_{n,+}$ are called \emph{even} chessboard elements,
those in $C_{n,-}$  \emph{odd} chessboard
elements. Informally,  in an even chessboard element all the
values  agree in parity with their positions. In an odd chessboard  element in  every position there is a value of opposite parity.
For $n=2m+1$, clearly $C_{n,\,-}=\emptyset$ and therefore $C_n = C_{n,\,+}$.

Note that the chessboard elements $C_n$ form a subgroup of $S_n$ and that
the even chessboard  elements $C_{n,\,+}$ form a subgroup of $C_n$.

The odd length is defined as follows (see also \cite{KV} and \cite{BC}).
\begin{defn}Let $n\in \mathbb N$ and $\sigma \in S_n$. The \emph{odd length} of $\sigma$ is
\begin{equation}\label{olA}
L(\sigma):= |\{(i,j)\in [n]^{2} : i<j,\,\sigma(i)>\sigma(j),\,i\not\equiv j \pmod{2}\}|.
\end{equation}
\end{defn}
Informally, the statistic $L$ counts inversions between values in
positions with opposite parity.  In the next proposition we collect some properties satisfied by $L$.
\begin{pro}\label{baspro}
Let $n\in \mathbb N$ and  let $w_0$ be the unique longest element of $S_n$. Then
\begin{enumerate}
\item[(i)] $L(e)=0$,
\item[(ii)] $L(s_i)=1,$ for $i=1,\ldots,n-1$,
\item[(iii)] $L(\sigma w_0)=L(w_0 \sigma)=L(w_0)-L(\sigma)$ for all $\sigma \in S_n$, 
\item[(iv)] $w_0$ is the unique element on which $L$ attains its
  maximum,  and $L(w_0)=\left\lfloor \frac{n}{2}\right\rfloor\left\lceil\frac{n}{2} \right\rceil$.
\end{enumerate}
\end{pro}
\begin{proof}
The only non-trivial point is the last one. It follows from (iii) and
the fact that the identity  is the unique element on which $L$ is zero. The last statement comes from the fact that, by definition, 
$L(w_0) = \sum_{i=1}^{n-1} \lceil \frac{i}{2} \rceil $.\qedhere
\end{proof}

The following result, conjectured in~\cite{KV} and proved
in~\cite{BC},  gives explicit product formulas for the sign-twisted
distribution of the odd length   on all quotients of the symmetric groups.

\begin{thm}
Let $n \in {\mathbb N}$, $I \subseteq [n-1]$, and $I_{1}, \ldots , I_{s}$ be the connected components of $I$. Then
 \begin{align}
 \sum_{\sigma \in S_{n}^{I}} (-1)^{\ell (\sigma )} x^{L(\sigma )} &=\left\{ \begin{array}{l}\mc 
{\displaystyle \prod_{k=b+1}^{\left\lfloor \frac{n-1}{2} \right\rfloor }}(1-x^{2k}), \\
 \mbox{if $n \equiv 1 \pmod{2}$, or if $n = 2 b$,} \\ 
 (1+x^{m}) \mc {\displaystyle \prod_{k=b+1}^{\left\lfloor \frac{n-1}{2} \right\rfloor }}(1-x^{2k}), \\
  \mbox{otherwise,}\end{array} \right.  
 \end{align}
where $b := \sum_{k=1}^{s} \left\lfloor \frac{|I_{k}|+1}{2} \right\rfloor $.
\end{thm}

More precisely, the following result is what is proved in \cite{BC}.

\begin{thm}\label{thmA}
Let $n\in {\mathbb N}$, $I \subseteq [n-1]$, and $I_{1},\ldots , I_{s}$ be the connected components of $I$. Then
\begin{align}
\sum_{\sigma \in \D_{\emptyset}^{I}(C_{n,\, +})}(-1)^{\ell (\sigma )}x^{L(\sigma )} &=
{\footnotesize\left[ \begin{array}{c}  b \\
\left\lfloor \frac{|I_{1}|+1}{2}\right\rfloor,\ldots,\left\lfloor \frac{|I_{s}|+1}{2}\right\rfloor 
\end{array} \right]_{x^2}} \prod_{k=b +1}^{\left\lfloor \frac{n-1}{2}\right\rfloor} (1-x^{2k}) \,,
\label{c2mp} 
\end{align}
and
\begin{align}
\sum_{\sigma \in \D_{\emptyset}^{I}(C_{n,\, -})}(-1)^{\ell (\sigma )}x^{L(\sigma )} &=
\begin{cases}
  0, & \mbox{if $n \equiv 1 \pmod{2} $, or} \\
  &  \mbox{if $n \equiv 0 \pmod{2} $ and $ b=m$}\\
-x^{m} \; {\displaystyle \sum_{\sigma \in \D_{\emptyset}^{I}(C_{n,\, +})}}(-1)^{\ell (\sigma )}x^{L(\sigma )}, & 
\mbox{otherwise,}
\end{cases}  
\label{c2mm} 
\end{align}
where $b := \sum_{k=1}^{s} \left\lfloor \frac{|I_{k}|+1}{2}
\right\rfloor$, and $m:=\lfloor \frac{n}{2}\rfloor$.
\end{thm}


\vspace{1em}

\section{Odd diagrams, odd inversion sets, and odd Schubert varieties}\label{sec:oddiag}

In this section we define and study odd analogues of familiar 
combinatorial and geometric objects which are associated to permutations.
More precisely, we define and study odd diagrams, odd Schubert varieties, and odd inversion sets. In particular, we
give a geometric interpretation of the odd length function
$L : S_n \rightarrow \mathbb N_0 $.

Let $n \in \mathbb N$ and $\sigma \in S_n$. 
We define the {\em odd diagram} of $\sigma$ to be 
$$
D_o(\sigma) := \{ (i,j) \in D(\sigma) : \sigma^{-1}(j) \not \equiv i \pmod{2} \}.
$$
Clearly $D_o(\sigma)\subset D(\sigma)$ for all $\sigma\in S_n$. Also, note that $|D_o(\sigma)| = L(\sigma)$.

Let 
$(U_1,\ldots , U_n) \in C(\sigma)$  and$(a_{i,j})_{(i,j) \in D(\sigma)} \in \CC^{D(\sigma)}$ 
be the corresponding set of complex numbers (see \S~2). 
We define the {\em odd Schubert cell} of $\sigma$ to be 
\[
C_o(\sigma) := \{ (U_1,\ldots , U_n) \in C(\sigma) : a_{i,j}=0 \mbox{ if } (i,j) \in D(\sigma) \setminus D_o(\sigma) \}.
\]
So if $(U_1,\ldots , U_n) \in C_o(\sigma)$ then there are 
$(a_{i,j})_{(i,j) \in D_o(\sigma)} \in \CC^{D_o(\sigma)}$ such that 
$$
U_k = 
\langle \{e_{\sigma(i)} + \sum_{\{ j : (i,j) \in D_o(\sigma) \} } a_{i,j} \, e_j  \}_{1 \leq i \leq k} \rangle
$$
for all $k \in [n]$, and all flags of this form are in $C_o(\sigma)$.
In particular, $C_o(\sigma)$ is isomorphic to an affine space of dimension~$L(\sigma)$.
We then define the {\em odd Schubert variety} $X_o(\sigma)$ associated with $\sigma$ to be the closure of 
$\pi(C_o(\sigma))$ in $\pi(F_n)$,
where $\pi$ is the embedding defined in \S~\ref{sec:pre}.
The next result then follows from standard facts (see, e.g., \cite[Chap.\ I, Ex.\ 2.17]{Har}).
\begin{pro}
Let $\sigma \in S_n$. Then
$X_o(\sigma)$ is a complex projective variety of dimension $L(\sigma)$.
\end{pro}

We illustrate our definitions with an example.
\begin{exm}Let $\sigma :=[ 4,2,3,1]$. Then $C_o(\sigma)$ consists
of all flags $(U_1,U_2,U_3,U_4) \in F_4$ for which there are complex numbers $a,b,c \in \CC$ such that
$
U_1 = 
\langle \{e_{4} +a e_2 + b e_1 \} \rangle
$,
$
U_2 = 
\langle \{e_{4} +a e_2 + b e_1, e_2 \} \rangle
$,
and 
$
U_3 = 
\langle \{e_{4} +a e_2 + b e_1, e_2, e_3 + c e_1 \} \rangle
$. 
The Pl\"{u}cker coordinates of these subspaces are, respectively, $(b,a,0,1)$, 
$(b,0,0,0,-1,0)$, and $(b,-c,0,1)$. The Segre embedding of this triple of points is,
after removing 0's,
\begin{equation}
\label{OddSchuex}
(b^3,-b^2,ab^2,-abc,b^2,-b,-cb^2,cb,-ab,ac,-bc,c,b^2,-b,ab,-a,b,-1).
\end{equation}
Therefore, the odd Schubert cell $C_o([4,2,3,1])$ may be identified with all the points in 
$\PP(\CC^{18})$ of the form (\ref{OddSchuex}) where $a,b,c \in \CC$.

\end{exm}
It is easy to characterize the permutations for which the diagram and the odd diagram coincide.
\begin{pro}\label{pro:coincide}
  Let $\sigma\in S_n$. We have $D(\sigma)=D_o(\sigma)$  if and only if $\sigma(k-2)<\sigma(k)$ for all $3\leq k\leq n$.
\end{pro}
The following is an immediate consequence.
\begin{cor}
  Let $\sigma\in S_n$. The odd Schubert variety $X_o(\sigma)$
  coincides with the  Schubert variety $X(\sigma)$ if and only if
  $\sigma(k-2)<\sigma(k)$  for all $k\in[3,n]$.
\end{cor}

  For $n\in\mathbb N$, there are
  $\displaystyle\binom{n}{\left\lfloor\frac{n}{2}\right\rfloor}$  
  permutations of degree  $n$ satisfying the last property, and thus
  for which $L$ and $\ell$  coincide; see also the
  sequence~\cite[A001405]{Oeis}. We write  $$G_n:=\{\sigma\in S_n :
  \sigma(k-2)<\sigma(k) \mbox{ for all }3\leq k\leq n\}=\{\sigma\in
  S_n : \ell(\sigma)=L(\sigma)\}$$  for the set of permutations for  which all inversions are odd inversions. 
In particular, the previous proposition implies that for permutations
in $G_n$ the odd diagram ``faithfully''  encodes the permutation
itself. It would be interesting to characterize,  for every $n$, the
largest subset of $S_n$ on which the  map associating to a permutation
its odd diagram is injective.  In \S~\ref{sec:conj} we put forward a conjecture  in
this direction. Now,  inspired by
Proposition~\ref{pro:coincide}, we prove a product formula  for the
distribution of the difference of  length and odd length, namely of
the {\em even inversions} $\einv$  over the symmetric groups.
\begin{pro}\label{pro:einv}
  Let $n\in\mathbb N$. Then
  \[\sum_{\sigma\in S_n}x^{\ell(\sigma)-L(\sigma)}=\sum_{\sigma\in
      S_n}x^{\einv(\sigma)}=
    \binom{n}{\left\lceil{\frac{n}{2}}\right\rceil} \prod_{i=1}^{n} 
    \left(\frac{1-x^{\lceil{\frac{i}{2}}\rceil}}{1-x}\right).\]
  \end{pro}
  \begin{proof}
    Consider the subgroups $$S_o:=\left\langle(i,i+2):i\equiv 1\!\pmod
      2\right\rangle\simeq  S_{\left\lceil\frac{n}{2}\right\rceil}$$
    and  $$S_e:=\left\langle(i,i+2):i\equiv 0\!\pmod 2\right\rangle\simeq S_{\left\lfloor{\frac{n}{2}}\right\rfloor}.$$
    
  Let $\pi_o$ and $\pi_e$ denote the natural projections from $S_n$ onto $S_o$
  and $S_e$,  respectively. It is easy to see that
  $\einv(\sigma)=\inv(\pi_o(\sigma))+\inv(\pi_e(\sigma))$.  Therefore,
  \begin{equation}\label{eq:poinceven}\sum_{\sigma\in S_o\times S_e}
    x^{\einv(\sigma)}= \left(\sum_{\sigma\in
        S_{\left\lceil{\frac{n}{2}}\right\rceil}}
      x^{\inv(\sigma)}\right)\left(\sum_{\sigma\in
        S_{\left\lfloor{\frac{n}{2}}\right\rfloor}}
      x^{\inv(\sigma)}\right)=\prod_{i=1}^{n}\left(\frac{1-x^{\lceil{\frac{i}{2}}\rceil}}{1-x}\right).
  \end{equation}
  The proposition follows, as $$S_n=\dot\bigcup_{\tau\in G_n} \tau (S_o\times S_e)$$
  and the identity in~\eqref{eq:poinceven} also holds on all of the
  $\displaystyle\binom{n}{\left\lfloor{\frac{n}{2}}\right\rfloor}$  cosets.
\end{proof}
The following is a straightforward corollary of this formula.
\begin{cor}
    The polynomial $\sum_{\sigma\in S_n}x^{\einv(\sigma)}$ is symmetric and unimodal for all $n\in \mathbb N$.
  \end{cor}

\begin{rmk}\label{rmk:modinv}
 The proof of Proposition~\ref{pro:einv} shows that similar results
 hold for the polynomials  giving the distribution of inversions between positions which are congruent modulo 
any positive integer. More precisely, for $k,n\in \mathbb N$ and
 $\sigma\in S_n$,  let $\inv_{k,0}(\sigma)$  denote the number of
 inversions between positions congruent  modulo~$k$ in $\sigma$,
\begin{equation}
\label{eq:zeromod}
      \inv_{k,0}(\sigma)=|\{(i,j)\in \Inv(\sigma) : \, j-i \equiv 0 \pmod k\}|.
\end{equation}
If $n=mk+r$ for some $m\in\mathbb N_0$ and $0 \leq r <k$ then 
\begin{eqnarray*}
\sum_{\sigma\in S_n}x^{\inv_{k,0}(\sigma)} &=&
\binom{n}{{\underbrace{\scriptsize m, \ldots ,m}_{k-r}},{\underbrace{\scriptsize m+1, \ldots ,m+1}_{r}}}\left( \sum_{\sigma\in S_m}x^{\inv(\sigma)} \right)^{k-r}
\left( \sum_{\sigma\in S_{m+1}} x^{\inv(\sigma)} \right)^{r} \\
&=& \binom{n}{{\underbrace{\scriptsize m, \ldots ,m}_{k-r}},{\underbrace{\scriptsize m+1, \ldots ,m+1}_{r}}}
\prod_{i=1}^{n} \left[ \left\lceil \frac{i}{k} \right\rceil  \right]_x.
\end{eqnarray*}
\end{rmk}

\vspace{3mm}

 It is clear that the odd Schubert variety attached to a permutation $\sigma \in S_n$ 
does not really depend on the permutation, but only on its odd diagram.
While it is well known that  diagrams are in bijection with
permutations,  this is not the case for odd diagrams. Indeed, already
in $S_3$  there are two permutations with the same odd diagram: $D_o([2,1,3])=\{(1,1)\}=D_o([3,1,2])$. 
It is therefore a natural and interesting problem to characterize
odd diagrams of permutations. This is probably not an easy task since
no characterization of (ordinary) diagrams seems to be known. 
Also, odd diagrams are probably not equivalent to any known combinatorial objects.
In fact, if we let $o_n$ be the number of different odd diagrams of permutations in $S_n$
then the first values of the sequence $\{o_n\}_{n\in \mathbb N}$ are $1,2,5,17,70,351,2041,13732$, and this sequence does not appear in the
OEIS (\cite{Oeis}).

It is easy to see that to characterize subsets of $[n]^2$ which are odd diagrams of a permutation it is enough to consider those that have at least one element in the first row or column.

\begin{pro}\label{pro:reduction}
  Let $S\subseteq [2,n]^2$. Then there exists $\sigma\in S_n$ such that $D_o(\sigma)=S$ if and only if there exists $\tau \in S_{n-1}$ such that $D_o(\tau)=\{(x-1,y-1): (x,y)\in S\}.$ 
\end{pro}
\begin{proof}
  Suppose there is $\sigma\in S_n$ such that $S=D_o(\sigma)$. Then $\sigma^{-1}(1)=1$ (otherwise 
  $(\sigma^{-1}(1)-1,1)\in D_o(\sigma)$). So $\tau=[\sigma(2)-1,\dots,\sigma(n)-1]\in S_{n-1}$ is the desired permutation.
  Conversely, if $\tau\in S_{n-1} $ is such that $D_o(\tau)=\{(x-1,y-1): (x,y)\in S\}$ then $\sigma=[1,\tau(1)+1,\dots,\tau(n-1)+1]$ has odd diagram equal to $S$.
  \end{proof}
  Note that the previous proof implies that if an odd diagram does not have any elements in the first column, then it has no elements in the first row.
  
The following result gives some necessary enumerative conditions for a subset of
$[n]^2$ to be the odd diagram of a permutation in $S_n$.
\begin{pro}\label{pro:oddiagnec}
Let $\sigma \in S_n$ and $S := D_o(\sigma)$. Then:
\begin{itemize}
\item[(i)] if $(i,j) \in S$ then $| \{ k \in [j-1] :  (i,k) \notin S \}| \leq 
\min\{\lfloor \frac{n+i-2}{2} \rfloor,j-1\}$;
\item[(ii)] if $(i,j) \in S$ then $| \{ k \in [i-1] :  (k,j) \notin S \}| \leq 
\min\{\lceil \frac{i+2j-3}{2} \rceil, i-1\}$;
\item[(iii)] if $i \in [n]$ then $| \{ j \in [n] :  (i,j) \in S \}| \leq \lceil \frac{n-i}{2} \rceil$;
\item[(iv)] if $j \in [n]$ then $| \{ i \in [n] :  (i,j) \in S \}| \leq 
  \min \{ \lceil \frac{n-1}{2} \rceil, n-j \}$.
\end{itemize}
\end{pro}
\begin{proof}

Suppose that $(i,j) \in S$. Let $k \in [j-1]$ be such that $(i,k) \notin S$.
Then either $\sigma^{-1}(k)<i$ or $\sigma^{-1}(k)>i$ and
$\sigma^{-1}(k) \equiv i \pmod{2}$.  But there are at most $i-1$
possibilities in  the first case, and at most $\lfloor \frac{n-i}{2}
\rfloor$ in the second case. This proves~(i).

Similarly, let $(i,j) \in S$, and $k \in [i-1]$ be such that $(k,j) \notin S$.
Then either $\sigma(k)<j$ or $ k \not \equiv i \pmod{2}$, and there are 
at most $j-1$ possibilities in the first case and at most 
$\lceil \frac{i-1}{2} \rceil$ in the second one. 

Finally, if $i,j \in [n]$ are such that $(i,j) \in S$ then $\sigma^{-1}(j) >i $ and 
$\sigma^{-1}(j) \not \equiv i \pmod{2}$, which proves (iii). The proof of (iv) is analogous and
is omitted.
\qedhere
\end{proof}

\begin{figure}
      \begin{tikzpicture}[scale=0.3]
        \draw (1,15) rectangle (15,16);
        \draw (1,13) rectangle (15,14);
        \draw (1,9) rectangle (15,10);
       \node [align=left] at (8.2,12) {${\bf \vdots}$ };
        \draw (16,3) rectangle (15,4);
        \draw (16,15) rectangle (15,16);
        \draw (16,13) rectangle (15,14);
        \draw (16,9) rectangle (15,10);
        \draw (1,3) rectangle (15,4);
        \node at (8,9.5) {$\cdots \star \cdots$};
        \node at (8,3.5) {$\cdots \star \cdots$};
        \node at (3+5,15.5) {$\cdots \star \cdots$};
        \node at (3+5,13.5) {$\cdots \star \cdots$};
        \node at (15.5,3.5) {$\star$};
        \node at (0,15.5) {\tiny $i_1$};
        \node at (0,13.5) {\tiny $i_2$};
        \node at (0,9.5) {\tiny $i_k$};
        \node at (0,3.5) {\tiny $i_{k+1}$};
        \draw (1.5,2 ) -- (1.5,3.5 );
        \node at (1.6,1.5) [align=left] {\tiny column $j$};
        \node at (26.3,2) [align=center]{ $\star$ \tiny$ =$ element of the   odd diagram};
        \draw [decorate,decoration={brace,amplitude=8pt},xshift=-4pt,yshift=0pt]
        (-0.88,3.5) -- (-0.88,15.5) node [text width=3cm,align=left,black,midway,xshift=-1.6cm] 
        { \tiny$ k+1$ congruent rows  with at least one star  in $[j,j+k-1]$ in the first $k$ rows};
         \draw [decorate,decoration={brace,amplitude=8pt},xshift=-4pt,yshift=0pt] (16.5,15.5)--(16.5,9.5)  node [text width=4cm,align=left,black,midway,xshift=3cm] {\tiny no star in positions \\ $(i_1,j+k),\dots,(i_k,j+k)$};
         \draw [decorate,decoration={brace,amplitude=8pt},xshift=0pt,yshift=-3pt]
        (1.5,17) -- (14.5,17) node [align=center,black,midway,yshift=.6cm] 
        { \tiny $ k$};
  \end{tikzpicture}
\caption{A forbidden configuration for odd diagrams}\label{fig:forbidden}
\end{figure}

The next proposition collects a number of configurations which cannot occur in odd diagrams (see also Figures~\ref{fig:forbidden} and \ref{fig:f2}).
\begin{pro}\label{pro:forbidden}
  Let $\sigma \in S_n$ and $S := D_o(\sigma)$. Then:
  \begin{itemize}
\item[(i)] if $(i,j),(k,l) \in S$ with $i \leq k$, $j \geq l$, and $i \equiv k \pmod{2}$ 
then $(i,l) \in S$;
\item[(ii)] if $(i,j),(k,l) \in S$ with $i<k$, $j \geq l$, and $i \not \equiv k \pmod{2}$ then $(i,l) \notin S$;
\item[(iii)] if $ \{ i_1, \ldots , i_{k+1} \}_< \subseteq [n]$ and $j \in [n]$ are 
such that $i_r \equiv i_{r+1} \pmod{2}$, 
$\{ i_r \} \times [j,j+k-1] \cap S \not = \emptyset$,
and $(i_r,j+k) \notin S$ for $r \in [k]$,
then $(i_{k+1},j+k) \notin S $;
\item[(iv)] if   $ \{ j_1, \ldots , j_{\lceil \frac{k+1}{2} \rceil +1} \}_< \subseteq [n]$, 
and $i_1, \ldots ,i_{\lceil \frac{k+1}{2} \rceil } \in [i,i+k-1]$ are such that 
$ (i_r,j_r) \in S$, $(i+k,j_r) \notin S $, and 
$i_r \equiv i+k \pmod{2}$ for $r=1, \ldots , \lceil \frac{k+1}{2} \rceil$,
 then $(i+k,j_{\lceil \frac{k+1}{2} \rceil +1}) \notin S $;
\item[(v)] if $j$ is the minimum index for which $([n]\times[j])\cap S\neq\emptyset$, then if $i\in [n]$ is such that $(i,j)\in S$ and  $(\{i+1\}\times [j+1,n])\cap S \neq \emptyset$ then $(i+2,j)\in S$;
\item[(vi)] if $(i,j)\in[n-2]\times [2,n-1] $ are such that $(i,j),(i+2,j-1)\in S$ and 
$(i+1,j),(i+2,j)\not\in S$ then $\{i+1\}\times[j+1,n]\cap S=\emptyset$.
\end{itemize}\end{pro}
\begin{proof}
We first prove~(i).
Since $(i,j),(k,l) \in S$ we have that $\sigma(i) >j$, $\sigma^{-1}(l) >k$,
and $\sigma^{-1}(l) \not \equiv k \pmod{2}$. Hence $\sigma^{-1}(l)
\not \equiv i \pmod{2}$ so $(i,l) \in D_o(\sigma)$.

The proof of~(ii) is identical, except that now $\sigma^{-1}(l) \equiv i \pmod{2}$ so $(i,l) \notin D_o(\sigma)$.

We now prove~(iii) (see Figure \ref{fig:forbidden}). 
Suppose, by contradiction, that $(i_{k+1},j+k) \in D_o(\sigma)$. Then 
$\sigma^{-1}(j+k) > i_{k+1}$ and $\sigma^{-1}(j+k) \not \equiv i_{k+1} \pmod{2}$.
Let $r \in [k]$. Since $ (i_r,j+k) \notin S$, by what we have just observed we have 
that $ \sigma(i_r) < j+k$. On the other hand , since 
$ \{ i_r \} \times [j,j+k-1] \cap S \not = \emptyset$, $ \sigma(i_r) > j$.
So $ \sigma(i_r) \in [j+1,j+k-1]$ for all $r \in [k]$, which is a contradiction.

To prove (iv) suppose, by contradiction, that $(i+k,j_{\lceil \frac{k+1}{2} \rceil +1}) \in S $. 
Then $\sigma(i+k) > j_{\lceil \frac{k+1}{2} \rceil +1}$.
Let $r \in [\lceil \frac{k+1}{2} \rceil ]$. Since $ (i_r,j_r) \in S$, we have that 
$ \sigma(i_r)> j_r$, $ \sigma^{-1}(j_r) \not \equiv i_r \pmod{2}$, and
$ \sigma^{-1}(j_r) > i_r$. So $\sigma^{-1}(j_r) \not \equiv i+k \pmod{2}$.
On the other hand, since $(i+k,j_r) \notin S$,
$ \sigma^{-1}(j_r) < i+k$. So $ \sigma^{-1}(j_r) \in [i+1,i+k-1]$ and 
$\sigma^{-1}(j_r) \not \equiv i+k \pmod{2}$
for all $r \in [\lceil \frac{k+1}{2} \rceil]$, which is a contradiction. This proves (iv).

Parts (v) and (vi) are easy to check (see Figure~\ref{fig:f2}).\qedhere
\end{proof}

Note that for $j=l$ part~(ii) of Proposition~\ref{pro:forbidden} implies that if $(i+1,j) \in S$ then $ (i,j) \notin S$.



In the following result, we collect a few more conditions satisfied by odd diagrams which 
say that some configurations can only appear in certain areas of the square grid.

\begin{pro}
\label{pro:moreoddiagnec}
Let $\sigma \in S_n$ and $S := D_o(\sigma)$. Then:
\begin{itemize}
\item[(i)] if $i,j,k \in [n]$ are such that $(i+2k-1,j+k-1) \in [n]^2$,  
$\{ i \} \times [j,j+k-1] \subseteq S$,
and $[i+1,i+2k-1] \times [j,j+k-1] \subseteq [n]^2 \setminus S$,
then $j \geq k$;
\item[(ii)] if $i,j \in [n]$, $k \in \N$ are such that $(i+2k,j+2k) \in [n]^2$,  
and 
\[
(i+a,j+b) \in S \Leftrightarrow  a \equiv b \equiv 0 \pmod{2}, \mbox{ and } a+b \leq 2k
\]
for all $ (a,b) \in [0,2k]^2$, 
then $i+j \geq k+2$.
\end{itemize}
\end{pro}
\begin{figure}
  \begin{tikzpicture}[scale=0.4]
  \node [align=left] at (1,9) {(a)};
 \node [align=left] at (18,9)
  {(b)};
       \draw [color=black] (6,2) grid (7,5);
        \draw (7,3) rectangle (14,4);
       \draw (3.5,2.5) -- (5.9,4.2);
        \node at (3,2) {\tiny $\star$ in $(i,j)$};
        \node at (6.5,4.5) {$\star$};
       \node at (10,3.5) {$\cdots \star \cdots$};
        \draw [decorate,decoration={brace,amplitude=5pt},xshift=-4pt,yshift=0pt]
        (7.3,5) -- (14,5) node [text width=3.5cm,black,midway,xshift=4mm,yshift=7.7mm] 
       { \tiny  at least one $\star$  in $[j+1,n]$ in  row $i+1$};
       \draw [decorate,decoration={brace,amplitude=3pt},xshift=0pt,yshift=0pt] (1,5)--(6-0.2,5)  node [text width=2cm,align=left,black,midway,yshift=5mm] {\tiny empty columns};

       \draw [color=black] (15+6,2) grid (15+7,5);
       \draw(20,2) rectangle (21,3);
         \node at (15+6.5,4.5) {$\star$};
       \node at (20.5,2.5) {$ \star $};
        \node at (15+10,2.5+1) { $\cdots \star \cdots$};
        \draw [decorate,decoration={brace,amplitude=5pt},xshift=-4pt,yshift=0pt]
        (15+7.3,4.5) -- (15+14,4.5) node [text width=3.8cm,black,midway,xshift=4.5mm,yshift=10mm]  { \tiny   $\star$ at $(i,j)$   and $(i+2,j-1)$ and at least one $\star$  in $[j+1,n]$ in row $i+1$};
       \draw (15+7,2+1) rectangle (15+14,3+1);
       \node at (15.3,0) [align=center]{ $\star$ \tiny$ =$ element of the   odd diagram};
  \end{tikzpicture}
\caption{Forbidden configurations for odd diagrams}\label{fig:f2}
\end{figure}

\begin{proof}
We first prove (i). Since $\{ i \} \times [j,j+k-1] \subseteq S$ we have that $\sigma^{-1}(j+r-1)>i$
and $\sigma^{-1}(j+r-1) \not \equiv i \pmod{2}$ for all $r \in [k]$.
Therefore there is $r_0 \in [k]$ such that $\sigma^{-1}(j+r_0-1) \geq i+2k-1$.
Hence, since $\{ i+2t \} \times [j,j+k-1] \subseteq [n]^2 \setminus S$
for all $t \in [k-1]$, 
$\sigma(i+2t) \leq j+k-1$ for all $t \in [k-1]$, 
so $\sigma(i+2t) < j$ for all $t \in [k-1]$, and the result follows.

We now prove (ii) (see Figure \ref{fig:grid}).
We show  that
\begin{equation}
\label{claim}
|\{ a \in [k] : \sigma(i+2a-1) < j \}  \cup 
\{ b \in [k] : \sigma^{-1}(j+2b-1) < i \}| \geq k,
\end{equation}
which implies our claim.

\begin{figure}
      \begin{tikzpicture}[scale=0.4]
       
        \draw[gray,dashed] (-0.9,1.2) grid(6.9,8.9);
        \draw[white,solid] (0,2) grid(6,8);
        \draw[black,solid] (0,2) grid(6,8);
        \node at (.5,7.5) {$ \star $};
        \node at (2.5,7.5) {$ \star $};
        \node at (4.5,7.5) {$ \star $};
        \node at (.5,5.5) {$ \star $};
        \node at (.5,3.5) {$\star$};
        \node at (-1,9) {\tiny $(i,j)$};
        \node at (2.5,5.5) { $\star$};
        \draw (-.7,8.7 ) -- (0.2,7.8 );
        \node at (4,0.6) [align=center]{ $\star$ \tiny$ =$ element of the   odd diagram};
      \end{tikzpicture}
\caption{This configuration can only appear if $i+j\geq 4$}\label{fig:grid}
\end{figure}

We proceed by induction on $k \geq 0$.
Let $k=1$. Assume, by contradiction, that (\ref{claim}) fails. Then $\sigma(i+1) \geq j$ and
 $\sigma^{-1}(j+1) \geq i$. Hence, since $(i+2,j) \in S$ and $(i,j+1) \notin S$, $\sigma(i+1) \geq j+2$.
 Similarly, since $(i,j+2) \in S$ and $(i,j+1) \notin S$, $\sigma^{-1}(j+1) \geq i+2$.
 But then either $(i,j+1) \in S$ (if $\sigma^{-1}(j+1) \not \equiv i \pmod{2}$) or
 $(i+1,j+1) \in S$ (if $\sigma^{-1}(j+1) \equiv i \pmod{2}$), which is a contradiction. 
 So assume $k \geq 2$. Since $(i,j+2k) \in S$
we have that $\sigma(i) > j+2k$, $\sigma^{-1}(j+2k)>i$, and $\sigma^{-1}(j+2k)
\not \equiv i \pmod{2}$. 

If $\sigma^{-1}(j+2k)=i+1$ then $\sigma^{-1}(j+2b-1)<i$
for all $b \in [k]$ (else either $(i,j+2b-1) \in S$ or $(i+1,j+2b-1) \in S$
for some $b \in [k]$) and the claim holds. 

Assume now that $\sigma^{-1}(j+2k)=i+2b+1$ for
some $b \in [k-1]$ then $j+2k-1=\sigma(i+2)$ (else 
$(i+2,j+2k) \in S$). Hence $\sigma(i+1)<j$ (for if $\sigma(i+1)>j+2k$ then
$(i+1,j+2k-1) \in S$, while if $j \leq \sigma(i+1) < j+2k$ then necessarily 
$\sigma(i+1)=j+2a-1$ for some $a \in [k-1]$, which implies that $(i,j+2a-1) 
\in S$, which again contradicts our hypotheses). In an analogous way one
concludes that $\sigma^{-1}(j+1)<i$ (for if $\sigma^{-1}(j+1) \geq i+2k$ then
either $(i,j+1) \in S $ or $(i+2b+1,j+1) \in S$, while if 
$i \leq \sigma^{-1}(j+1) < i+2k$ then $\sigma^{-1}(j+1) \not \equiv i
\pmod{2}$, so again $(i,j+1) \in S$). Now, by our induction hypothesis
(applied to $i+2,j+2,k-2$) we conclude that 
\[
|\{ a \in [k-2] : \sigma(i+2a+1) < j+2 \}  \cup 
\{ b \in [k-2] : \sigma^{-1}(j+2b+1) < i+2 \}| \geq k-2.
\]
But if $\sigma(i+2a+1) < j+2$ for some $a \in [k-2]$ then 
$\sigma(i+2a+1) < j$ (for if $\sigma(i+2a+1) = j+1$ then $(i,j+1) \in S$,
while $\sigma(i+2a+1) \neq j$ since $(i+2k,j) \in S$). Also, since $\sigma(i+1)<j$,
and $\sigma(i)>j+2k$, if $\sigma^{-1}(j+2b+1) < i+2$ for some $b \in [k-2]$, 
then $\sigma^{-1}(j+2b+1) < i$. Therefore 
\[
|\{ a \in [k-2] : \sigma(i+2a+1) < j \}  \cup 
\{ b \in [k-2] : \sigma^{-1}(j+2b+1) < i \}| \geq k-2,
\]
and this implies (\ref{claim}) since $\sigma(i+1)<j$ and 
$\sigma^{-1}(j+1)<i$.

Finally, assume that $\sigma^{-1}(j+2k)=i+2b+1$ for
some $b \geq k$. Then as in the previous case we conclude that 
$j+2k-1=\sigma(i+2)$ and $\sigma(i+1)<j$. Therefore $\sigma(i+4)=j+2k-3$
(for if $\sigma(i+4)> j+2k$ then $(i+4,j+2k) \in S$, while if $\sigma(i+4)=
j+2k-2$ then $(i+2,j+2k-2) \notin S$). This implies that $\sigma(i+3)<j$
(for if $\sigma(i+3)>j+2k-3$ then $(i+3,j+2k-3) \in S$). Now, by our induction
hypothesis (applied to $i+4,j,k-2$) we have that 
\[
|\{ a \in [k-2] : \sigma(i+2a+3) < j \}  \cup 
\{ b \in [k-2] : \sigma^{-1}(j+2b-1) < i+4 \}| \geq k-2.
\]
But if $\sigma^{-1}(j+2b-1) < i+4$ for some $b \in [k-2]$ then
$\sigma^{-1}(j+2b-1) < i$ so
\[
|\{ a \in [k-2] : \sigma(i+2a+3) < j \}  \cup 
\{ b \in [k-2] : \sigma^{-1}(j+2b-1) < i \}| \geq k-2
\]
and this proves (\ref{claim}) since $\sigma(i+1)<j$ and $\sigma(i+3)<j$.

This concludes the induction step and hence the proof.
\end{proof}

The conditions in Propositions~\ref{pro:oddiagnec},  \ref{pro:forbidden} and 
\ref{pro:moreoddiagnec}  are also sufficient for $S \subseteq [n]^2$, with  
$n\leq 4$, to be the odd diagram of a permutation. However, for $n\geq 5$ 
they fail to characterize these subsets. For instance, $\{ (1,1),(1,2),(3,2),(4,4) \}$ is not the 
odd diagram of any permutation.

\vspace{5mm}

While we are unable to characterize odd diagrams, we can characterize a
closely related set. 
For $\sigma \in S_n$ let 
$$
\Inv_o(\sigma) := \{ (i,j) \in \Inv(\sigma) : j \not \equiv i \pmod{2} \}.
$$
We call $\Inv_o(\sigma)$ the {\em odd inversion set} of $\sigma$.
Note that $|\Inv_o(\sigma)| = L(\sigma)$ and that $\Inv_o(\sigma)\subset OS_n$
where $OS_n := \{ (i,j) \in [n]^2 : i<j, j \not \equiv i \pmod{2}  \}$
is the {\em odd staircase} of size $n$.

The following result follows easily from our definitions.
\begin{lem}
Let $\sigma \in S_n$ and $(i,j) \in [n]^2$. Then
$(i,\sigma(j)) \in D_o(\sigma)$ if and only if $(i,j) \in \Inv_o(\sigma)$.
\end{lem}

We now characterize the odd inversion sets of permutations.
Recall that the {\em Tur\'an graph} (see, e.g., \cite{Turan}) is the complete
bipartite graph $T_n := ([n], E_n)$ where, if $i,j \in [n]$, then
$\{ i,j \} \in E_n$ if and only if $i \not \equiv j \pmod{2}$.

Given $I \subseteq OS_n$ we define an orientation $A_I$ of 
$T_n$ as follows. Let $\{ i,j \}_< \in E_n$. Then we let
$i \rightarrow j$ in $A_I$ if and only if $(i,j) \in I$.
So, for example, for $I=\{(1,4),(2,3),(2,5),(3,4)\}$ we get the orientation of $T_5$ in Figure~\ref{fig:Turan}.
We then have the following simple characterization of odd inversion sets
of permutations in terms of orientations of $T_n$.

\begin{pro}
Let $I \subseteq OS_n$. Then there is a permutation $\sigma \in S_n$ 
such that $I=\Inv_o(\sigma)$ if and only if $A_I$ is acyclic.
\end{pro}
\begin{proof}
Suppose first that $I=\Inv_o(\sigma)$ for some $\sigma \in S_n$. Then we have that,
for all $\{ i,j \}_< \in E_n$, $i \rightarrow j$ in $A_I$ if and only if
$\sigma(i) > \sigma(j)$, so $A_I$ is acyclic.

Conversely, suppose that $A_I$ is acyclic. It is then easy to see, by
induction on the number of  vertices, that given any acyclic
orientation of a 
graph $D=(V,E)$ there is a bijection $f: V \rightarrow [|V|]$ such that
if $\{ x,y \} \in E$ then $x \rightarrow y$ if and only if $f(x)>f(y)$. Indeed,
as the orientation is acyclic there is either a source or a sink $v \in V$. 
Say $v$ is a source. Now 
define $f(v):= |V|$, remove $v$ and all edges incident to it from $D$ and argue
by induction. In particular, there is $\sigma \in S_n$ such that 
$i \rightarrow j$ if and only if
$\sigma(i) > \sigma(j)$ for all $\{ i,j \}_< \in E_n$.
So $I=\Inv_o(\sigma)$.
\end{proof}
\begin{figure}
  \begin{tikzpicture}
    \tikzset{vertex/.style = {shape=circle,fill=black,inner sep=1.5pt}}
\tikzset{edge/.style = {->,thick,> = latex'}}
\node[vertex,label={south  west:$5$}] (a) at  (0,0) {};
\node[vertex,label={  west:$3$}] (b) at  (0,2) {};
\node[vertex,label={north  west:$1$}] (c) at  (0,4) {};
\node[vertex,label={south  east:$4$}] (d) at  (2,1) {};
\node[vertex,label={north  east:$2$}] (e) at (2,3) {};
\draw[edge] (b) to (d);
\draw[edge] (a) to (d);
\draw[edge] (e) to (b);
\draw[edge] (e) to (c);
\draw[edge] (c)  to (d);
\draw[edge] (e)  to (a);

  \end{tikzpicture}\caption{Ayclic orientation of $T_5$ defined by   $I=\{(1,4),(2,3),(2,5),(3,4)\}$}\label{fig:Turan}
  \end{figure}

We illustrate the preceding result with an example.
\begin{exm} Given $\sigma=[3,5,4,1,2]\in S_5$ we have
  $\Inv_o(\sigma)=\{(1,4),(2,3),(2,5),(3,4)\}$,  which defines the
  acyclic orientation in Figure~\ref{fig:Turan}.  Conversely, given
  $I=\{(1,4),(2,3),(2,5),(3,4)\}$, following the steps of the
  induction  and maintaining notation from the above proof we get:
  $f(2)=5$, $f(1)=4$,  $f(3)=3$, $f(5)=2$ and $f(4)=1$, which defines
  the permutation  $\tau=[4,5,3,1,2]$ with $\Inv_o(\tau)=I$. As
  expected,  this is not the only permutation of $S_5$ with this odd inversion
  set.  There are $6$ permutations with odd inversion set equal to
  $I$:  $[2,5,3,1,4]$,  $[2,5,4,1,3]$,  $[3,5,2,1,4]$,  $[4,5,2,1,3]$, 
  $\sigma$ and $\tau$. 
  \end{exm}

\vspace{1em}

\section{Shifting and  reversing}\label{sec:operations}

In this section we derive a number of results concerning
operations that can  be performed on the subsets defining a descent
class,  after  which the sign-twisted generating function of the odd length
remains the same  or changes in a controlled way.
We also give
sufficient conditions on a descent class for the corresponding sign-twisted
generating function to be zero, and we compute it explicitly for the
descent  class of the alternating permutations.

Recall that a permutation in the descent class $\D_J^I(S_n)$ is a
permutation  which is increasing in the positions corresponding to
 $I\cup (I+1)$  and decreasing in  $J\cup (J+1)$.

The proofs of the following two results are similar to  those
of~\cite[Lemma~3.1 and Proposition~3.3]{BC}. 
However, for the reader's convenience, and for completeness, we
provide proofs here.
Our first lemma shows that the sign-twisted generating function of the odd length is zero on the non-chessboard elements of a descent class in which the ascents and the descents are disjoint.
\begin{lem}\label{chdes}
Let $I,J\subseteq [n-1]$, $I \cap J = \emptyset$. Then
\[\sum_{\sigma \in \D_J^I(S_n)} \f =\sum_{\sigma \in \D_J^I(C_n)} \f.\]
\end{lem}
\begin{proof}
Let $\sigma \in\D_J^I(S_n)\setminus\D_J^I(C_n)$. Then there exists 
$i\in[n-1]$ such that $\sigma^{-1}(i)\equiv \sigma^{-1}(i+1) \pmod 2$
(else either $\sigma ^{-1}(i) \equiv i \pmod{2}$ for all $i \in [n]$ or 
$\sigma ^{-1}(i) \equiv i+1 \pmod{2}$ for all $i \in [n]$ so $\sigma \in C_n$).
Let $i$ be minimal with this property and define $\sigma^{*}=s_{i} 
\sigma$. This is a well defined involution on 
$\D_J^I(S_n)\setminus\D_J^I(C_n)$ since 
$|\sigma^{-1}(i)-\sigma^{-1}(i+1)|\geq 2$. But $L(\sigma^{*})=L(\sigma)$ and 
$\ell(\sigma^{*})=\ell(\sigma) \pm 1$, which implies the result.\qedhere
\end{proof}
The next result is the first of a series of invariance results for the sign-twisted generating function of the odd length over a descent class $\D ^I_ J(S_n)$. It shows that a connected component of odd cardinality of the  ascents can be shifted or enlarged of one unit to the right  without changing the generating function, as long as it remains a connected component.
\begin{pro}
\label{scr}
Let $I,J \subseteq [n-1]$, $I \cap J = \emptyset$. Let $i \in {\mathbb N}$, $k \in {\mathbb N}_0$ be such that
$[i,i+2k]\subseteq I$ is a connected component of $I \cup J$ and $i+2k+2 \not \in I\cup J$. 

Then
\begin{equation}
\label{sc}
\sum_{\sigma \in  \D ^I_ J(S_n)}(-1)^{\ell (\sigma )}x^{L(\sigma )} =
\sum_{\sigma \in  \D ^{I\cup \tilde{I}}_ J(S_n)}(-1)^{\ell (\sigma )}x^{L(\sigma )}   =
\sum_{\sigma \in  \D ^{ \tilde{I}}_ J (S_n)}(-1)^{\ell (\sigma )}x^{L(\sigma )}
\end{equation}
where $\tilde{I} := (I \setminus \{ i \} ) \cup \{ i+2k+1 \} $.
\end{pro}
\begin{proof} First note that, by our hypotheses, $(I \cup \tilde{I}) \cap J=\emptyset $. 
We have 
\begin{eqnarray} \nonumber
\sum_{\sigma \in \D_J^I(S_n)}(-1)^{\ell (\sigma )}x^{L(\sigma )} & = &\hspace{-1em}\sum_{\substack{\sigma \in \D_J^I(S_n): \\ \sigma (i)>\sigma (i+2k+2)}  }\hspace{-2em}(-1)^{\ell (\sigma )}x^{L(\sigma )} + \sum_{\substack { \sigma \in \D_J^I(S_n): \, \sigma (i+2k+1) <\\  \, \sigma (i+2k+2)}}\hspace{-1em}(-1)^{\ell (\sigma )}x^{L(\sigma )} \\ \label{sum}& + & \sum_{j=1}^{2k+1} \left(  \sum_{\substack {\sigma \in \D_J^I(S_n): \; \sigma (i+j-1)< \\ \sigma (i+2k+2)<\sigma (i+j) }}(-1)^{\ell (\sigma )}x^{L(\sigma )} \right).
\end{eqnarray}
Let $r \in [k]$. Note that, by our hypotheses, $ i-1 \notin J$ and
$ i+2k+1 \notin J$. Therefore the map $\sigma \mapsto \tilde{\sigma} := \sigma \; (i+2k+2 \, , i+2r) $ is a bijection between 
$\{ \sigma \in \D_J^I(S_n): \; \sigma (i+2r)< \sigma (i+2k+2)<\sigma (i+2r+1)\}$
and
$\{ \sigma \in \D_J^I(S_n): \; \sigma (i+2r-1)< \sigma (i+2k+2)<\sigma (i+2r)\}$.
Furthermore, $\ell (\widetilde{\sigma})=\ell (\sigma )+1$
and $L(\widetilde{\sigma}) =L(\sigma )$ so
\[
\sum_{\substack{\sigma \in  \D_J^I(S_n)\,:\, \sigma (i+2r) < \\  \sigma (i+2k+2) < \sigma (i+2r+1) }}(-1)^{\ell (\sigma )}x^{L(\sigma )} = -
\sum_{\substack{ \sigma \in  \D_J^I(S_n)\,:\, \sigma (i+2r-1) < \\  \sigma (i+2k+2) < \sigma (i+2r) }}(-1)^{\ell (\sigma )}x^{L(\sigma )}.
\]
Similarly, the bijection $\sigma \mapsto \sigma \; (i+2k+2 , i) $ shows that
\[\sum_{\substack{\sigma \in  \D_J^I(S_n)\,:\\ \, \sigma (i+2k+2) < \sigma (i) }}(-1)^{\ell (\sigma )}x^{L(\sigma )} =
- \sum_{\substack{\sigma \in  \D_J^I(S_n)\,:\, \sigma (i) < \\  \sigma (i+2k+2) < \sigma (i+1) }}(-1)^{\ell (\sigma )}x^{L(\sigma )}.
\]
Therefore, by (\ref{sum}), 
\[ \sum_{\sigma \in \D_J^I(S_n)}(-1)^{\ell (\sigma )}x^{L(\sigma )} =
\sum_{\substack{\sigma \in \D_J^I(S_n)\,:\, \sigma (i+2k+1) < \\ \sigma
    (i+2k+2)}} (-1)^{\ell (\sigma )}x^{L(\sigma )}
\]
and the first equality in (\ref{sc}) follows. 

The proof of the second equality is similar, and is therefore omitted.\qedhere
\end{proof}

Note that the proof of the previous result actually yields that 
if $I,J \subseteq [n-1]$ are such that $I \cap J = \emptyset$,  and
if  $i \in {\mathbb N}$ and $k\in \mathbb N _0$ are 
such that $[i,i+2k+1]$ is a connected component of $I \cup J$ and $i+2k+1 \in J$,
$[i,i+2k] \subseteq I$, then $\tilde{I}\cap J\neq \emptyset$, hence
\begin{equation*}
\sum_{\sigma \in  \D ^I_J(S_n)}(-1)^{\ell (\sigma )}x^{L(\sigma )} = 0.
\end{equation*}
This is a special case of a more general fact (see Proposition \ref{zero}).

The following is the ``left'' version of Proposition~\ref{scr}. Informally, it shows that a connected component of odd cardinality of the  ascents can be shifted or enlarged of one unit to the left  without changing the sign-twisted generating function, as long as it remains a connected component. 

\begin{pro}\label{scl}
Let $I,J \subseteq [n-1]$, $I \cap J =\emptyset$. Let $i \in {\mathbb N}$, $k \in {\N_0}$ be such
that $[i+1,i+2k+1] \subseteq I$ is a connected component of $I \cup J$, 
and $i-1 \not \in I\cup J$.
Then
\[ 
\sum_{\sigma \in \D ^I_ J(S_n)}(-1)^{\ell (\sigma )}x^{L(\sigma )} 
= \sum_{\sigma \in \D^{I \cup \bar{I}}_J (S_n)}(-1)^{\ell (\sigma )}x^{L(\sigma )}
= \sum_{\sigma \in \D^{\bar{I}}_J(S_n)}(-1)^{\ell (\sigma )}x^{L(\sigma )} 
\]
where $\bar{I} := (I \setminus \{ i+2k+1 \} ) \cup \{ i \} $.
\end{pro}
\begin{proof} 
Under our hypotheses we have that $(I \cup \bar{I}) \cap J=\emptyset $, $[i,i+2k]$ is a connected 
component of $\bar{I} \cup J$, $[i,i+2k] \subseteq \bar{I}$, and $i+2k+2 \not \in \bar{I} \cup J$, 
so the result follows from Proposition~\ref{scr}. \qedhere
\end{proof}

We now show that a connected component of even cardinality of the descents can be ``transformed'' (or ``reversed'') into a connected component of the ascents, by changing the generating function by a simple factor.
\begin{lem}\label{destoas}
Let $I,J \subseteq [n-1]$, $I \cap J = \emptyset$, and $i,k$ $\in \mathbb{N}$ be such that $K:=[i,i+2k-1]$ is a connected component of $I \cup J$, $K \subseteq J$. Then
\begin{equation} \sum_{\sigma \in {\mathcal D}^{I}_{J}(C_{n,\pm})}{(-1)^{\ell(\sigma)}x^{L(\sigma)}}=(-1)^k x^{k(k+1)}\sum_{\sigma \in {\D}^{I\cup K}_{J\setminus K}(C_{n,\pm})}{(-1)^{\ell(\sigma)}x^{L(\sigma)}}.
\end{equation}
In particular,
\begin{equation}
\sum_{\sigma \in {\D}^{I}_{J}(S_{n})}{(-1)^{\ell(\sigma)}x^{L(\sigma)}}=
(-1)^k x^{k(k+1)}\sum_{\sigma \in {\D}^{I\cup K}_{J\setminus K}(S_{n})}{(-1)^{\ell(\sigma)}x^{L(\sigma)}}.
\end{equation}
\end{lem}
\begin{proof}
We have  
\[ \sum_{\sigma \in \D^{I}_{J}(C_{n,+})}{(-1)^{\ell (\sigma )}x^{L(\sigma )}}=\sum_{\tau\in \D ^{I\cup K}_{J\setminus K}(C_{n,+})}{(-1)^{\ell (\bar{\tau} )}x^{L(\bar{\tau} )}},\]
where $\bar{\tau}:=[\tau(1),\ldots,\tau(i-1),\tau(i+2k),\ldots,\tau(i+1),\tau(i),\tau(i+2k+1),\ldots,\tau(n)]$. But $\ell(\bar\tau)=\ell(\tau)+ (2k+1)k$  and, by Proposition~\ref{baspro} $ L(\bar{\tau})=L(\tau)+k(k+1)$, thus
\[\sum_{\tau\in \D ^{I\cup K}_{J\setminus K}(C_{n,+})}{(-1)^{\ell (\bar{\tau} )}x^{L(\bar{\tau} )}}= (-1)^k x^{k(k+1)}\sum_{\tau\in \D ^{I\cup K}_{J\setminus K}(C_{n,+})}{(-1)^{\ell (\tau )}x^{L(\tau )}}\]
as desired. Similarly for $C_{n,-}$.\qedhere
\end{proof}

In a similar way, it is easy to determine the generating function on the descent class obtained by transforming all the descents  into ascents, and conversely, as shown in the following result.
\begin{pro}\label{destoas2}
Let $I,\,J \subseteq [n-1]$, $I \cap J = \emptyset$. Then
\[
\sum_{\sigma \in \D^{I}_{J}(S_n)}{(-1)^{\ell(\sigma)}x^{L(\sigma)}}=(-1)^{\ell(w_0)}x^{L(w_0)} \sum_{\sigma \in \D^{J}_{I}(S_n)}{(-1)^{\ell(\sigma)}\left(\frac{1}{x}\right)^{L(\sigma)}}.\]
\end{pro}
\begin{proof}
It is clear that the map $\sigma \mapsto w_0 \sigma$ is a bijection from $\D ^{I}_{J}(S_n)$ to $\D^{J}_{I}(S_n)$. Therefore, by Proposition~\ref{baspro} we have
\begin{align*}
\sum_{\sigma \in \D^{I}_{J}(S_n)}\!{(-1)^{\ell(\sigma)}x^{L(\sigma)}}&= \sum_{\tau \in \D^{J}_{I}(S_n)}{(-1)^{\ell(w_0 \tau)}x^{L(w_0 \tau)}}\\
&=(-1)^{\ell(w_0)}x^{L(w_0)}\sum_{\tau \in \D^{J}_{I}(S_n)}{(-1)^{\ell(\tau)}\left(\frac{1}{x}\right)^{L(\tau)}}.\end{align*}\end{proof}
\begin{rmk}\label{rmk:res}
The bijection $\sigma \mapsto w_0 \sigma$ in the proof of Proposition~\ref{destoas2} restricts to a bijection between chessboard elements of the relevant descent classes. In particular, if $n$ is even it is a bijection between $\D^I _J (C_{n,+})$ and  $\D^J _I (C_{n,-})$.
\end{rmk}

The sign-twisted generating function is also invariant under left and right shifting of connected components of the descents, under certain hypotheses. The next two results are analogous to Proposition~\ref{scr} and \ref{scl}, respectively. The first shows that a connected component of odd cardinality of the descents can be shifted (or enlarged of one unit) to the right, as long as it remains a connected component.

\begin{pro}\label{shrdes}
Let $I,\,J\subseteq [n-1]$, $I \cap J = \emptyset$. Let $i \in \mathbb N $, $k \in \N_0$ be such that $[i,i+2k]$ is a connected component of $I \cup J$, $[i,i+2k] \subseteq J$, and $i+2k+2 \notin I\cup J$. Then
\[\sum_{\sigma \in	\D^{I}_{J}(S_n)}{(-1)^{\ell(\sigma)}x^{L(\sigma)}}=\sum_{\sigma \in \D^{I}_{J\cup \tilde{J}}(S_n)}{(-1)^{\ell(\sigma)}x^{L(\sigma)}}=\sum_{\sigma \in \D^{I}_{\tilde{J}}(S_n)}{(-1)^{\ell(\sigma)}x^{L(\sigma)}},\]
where $\tilde{J}:=(J\setminus\{i\})\cup\{i+2k+1\}$.
\end{pro}
\begin{proof}
By Proposition~\ref{scr} we have 
\[\sum_{\sigma \in \D^{J}_{I}(S_n)}{(-1)^{\ell(\sigma)}x^{L(\sigma)}}=\sum_{\sigma \in \D^{J\cup \tilde{J}}_{I}(S_n)}{(-1)^{\ell(\sigma)}x^{L(\sigma)}}=\sum_{\sigma \in \D^{\tilde{J}}_{I}(S_n)}{(-1)^{\ell(\sigma)}x^{L(\sigma)}}\]
so the result follows from Proposition~\ref{destoas2}.\qedhere
\end{proof}

In a similar way, using Proposition~\ref{scl}, we obtain the following invariance result under left shifting of a connected component of odd cardinality of the descents.
\begin{pro}\label{shldes}
Let $I,\,J \subseteq [n-1]$, $I \cap J = \emptyset$, and $i \in \mathbb N$, $k \in \N_0$ be such that 
$[i+1,i+2k+1]$ is a connected component of $I \cup J$, $[i+1,i+2k+1] \subseteq J$, and 
$i-1 \notin I\cup J$. Then
\[\sum_{\sigma \in	\D^{I}_{J}(S_n)}{(-1)^{\ell(\sigma)}x^{L(\sigma)}}=\sum_{\sigma \in \D^{I}_{J\cup \bar{J}}(S_n)}{(-1)^{\ell(\sigma)}x^{L(\sigma)}}=\sum_{\sigma \in \D^{I}_{\bar{J}}(S_n)}{(-1)^{\ell(\sigma)}x^{L(\sigma)}},\]
where $\bar{J}:=(J\setminus\{i+2k+1\})\cup\{i\}$.
\end{pro}

Computer calculations suggest that the operation of shifting can be performed under weaker hypotheses, namely even if the connected component to be shifted is not contained in $I$ (as required in Proposition~\ref{scr}) and therefore not contained in $J$ (as in Proposition~\ref{shrdes}). 
More precisely, we conjecture the following.
\begin{con}
	Let $I,J \subseteq [n-1]$, $I \cap J = \emptyset$. Let $i \in {\mathbb N}$, $k \in {\mathbb N_0}$ be such that $i+2k+2 \not \in I\cup J$ and
	$[i,i+2k]$ is a connected component of $I \cup J$, say $[i,i+2k]=A \cup B$, where $A \subseteq I$ and $B \subseteq J$. 
Then
	$$
	\sum_{\sigma \in  \D ^I_ J(S_n)}(-1)^{\ell (\sigma )}x^{L(\sigma )} =	\sum_{\sigma \in  \D ^{ \tilde{I}}_ {\tilde{J}} (S_n)}(-1)^{\ell (\sigma )}x^{L(\sigma )}
	$$
where $\tilde{I}:=(I\setminus A) \cup(A +1)$ and $\tilde{J}:=(J\setminus B) \cup (B +1)$.	
\end{con}

\vspace{1em}

\section{Descent classes}\label{sec:unmixed}

In this section we investigate the sign-twisted generating function of the odd length over descent classes.
More precisely, we give sufficient conditions on a descent class for the generating function to be zero,
and we compute it explicitly for the alternating permutations and for a general family of descent
classes which includes all quotients.

Let $I,J \subseteq [n-1]$, $I \cap J = \emptyset$, and $i \in [n]$. We say that $i$ is a 
{\em peak} of $ \D ^I_ J(S_n) $ if  $ i \in (I+1) \setminus I $ or $ i \in J \setminus (J+1) $.
Similarly, $ i $ is a \textit{valley} if 
 $ i \in I \setminus (I+1) $ or $ i \in (J+1) \setminus J $.

\begin{pro}
\label{zero}
Let $I,J \subseteq [n-1]$, $I \cap J = \emptyset$, and $i \in {\mathbb N}$, $k \in {\mathbb N_0}$ be
such that $[i,i+2k+1]$ is a connected component of $I \cup J$ and $ v \not \equiv p \pmod{2} $ 
for any $v,p \in [i,i+2k+2]$, $v$ valley, $p$ peak. Then 
\[
\sum_{\sigma \in  \D ^I_ J(S_n)}(-1)^{\ell (\sigma )}x^{L(\sigma )} = 0.
\]
\end{pro}
 \begin{proof}
Let $ \sigma \in  \D ^I_ J(S_n)$. Let $ \{ a_1, \ldots , a_{2k+3}
\}_{<} := \{ \sigma(i), \sigma(i+1), \ldots , \sigma(i+2k+2) \} $. Let
$ v := \sigma^{-1}(a_1). $  Then $v$ is a valley (for if $i<v<i+2k+2$ then  
$ \sigma(v-1) > \sigma(v) < \sigma(v+1) $ so $ v \in I \cap (J+1) $,
while  if $v=i$ then 
$\sigma(v) < \sigma(v+1) $ so $v \in I \setminus (I+1)$, and if $v=i+2k+2$ then 
$\sigma(v-1) > \sigma(v)$ so $ v \in (J+1) \setminus J$).
Similarly,  $\sigma^{-1}(a_{2k+3})$ is a peak. Therefore, by our
hypotheses,  $ \sigma^{-1}(a_1) \not \equiv \sigma^{-1}(a_{2k+3}) \pmod{2}$.

Let $j := \mbox{min} \{ r \in [2k+2] : \sigma^{-1}(a_r) \equiv \sigma^{-1}(a_{r+1}) \pmod{2} \}$ (note that
$j$ certainly exists for if $\sigma^{-1}(a_1) \not \equiv \sigma^{-1}(a_2) \not \equiv \cdots  \not \equiv\sigma^{-1}(a_{2k+3}) \pmod{2}$ then  $ \sigma^{-1}(a_1) \equiv \sigma^{-1}(a_{2k+3}) \pmod{2}$ which
is a contradiction), and $\hat{\sigma} := (a_j, a_{j+1}) \, \sigma$. Then $\hat{\sigma} \in \D ^I_ J(S_n)$,
$\ell(\hat{\sigma})=\ell(\sigma) \pm 1$, $L(\hat{\sigma})=L(\sigma)$ and the map $\sigma \mapsto \hat{\sigma}$ is an involution. The result follows. \qedhere\end{proof}

Note that the converse of the previous result does not hold. For example, 
if $n=8$, $I=\{1,2,4\}$, and $J=\{3,5,6\}$  then the sign-twisted generating function for
$\D ^I_ J(S_8)$ is zero but $\D ^I_ J(S_8)$ has peaks $\{3,5\}$ and valleys $\{1,4,7\}$.
On the other hand, under the weaker hypothesis that there exist at
least one peak and one valley  with different parities the generating
function is not, in general, zero.  For example, if  
 $n=8$, $I=\{1,2,4\}$, and $K=\{3,5,6,7\}$ then $\D ^I_ K(S_8)$ has peaks $\{3,5\}$ and valleys $\{1,4,8\}$
 but the corresponding generating function is $-x^6(1+x^2+x^4)$.
It would be interesting to find necessary and sufficient conditions on
$I$ and $J$  for the sign-twisted generating function on $\D ^I_ J(S_n)$ to be zero. 

Proposition~\ref{zero} implies that if $I \cup J$ has a ``zig-zag''
connected component $K$  of even cardinality 
(i.e., if all even elements of $K$ are in $I$ and all odd ones are in
$J$, or conversely)  then 
the corresponding sign-twisted generating function is zero. Thus, this is in
particular  true for the alternating permutations of a symmetric group
of odd degree.  This makes it natural to investigate the corresponding
generating function  for all alternating permutations. For $n \in {\mathbb N}$ we let 
\[
E_n^{-} := \{ \sigma \in S_n : \, \sigma(1)>\sigma(2)<\sigma(3)> \cdots  \},
\]
and
\[
E_n^{+} := \{ \sigma \in S_n : \, \sigma(1)<\sigma(2)>\sigma(3)< \cdots  \}.
\]
We call the elements of $E_n^{-}$ (resp. $E_n^{+}$)  \textit{alternating} 
(resp.\ \textit{reverse alternating}) permutations
(we refer the reader to, e.g., \cite[\S 1.6]{StaEC2} for further information about alternating permutations).

\begin{pro}
Let $n \in {\mathbb N}$. Then 
\begin{align}
\sum_{\sigma \in E^{-}_{n}}(-1)^{\ell (\sigma )}x^{L(\sigma )} &=
\left\{ \begin{array}{ll}
0, & \mbox{if $n \equiv 1 \pmod{2}$}, \\ 
(-x)^{\frac{n}{2}}, & \mbox{if $n \equiv 0 \pmod{2}$,}
\end{array} \right.  
\label{altminus} 
\end{align}
and 
\begin{align}
\sum_{\sigma \in E^{+}_{n}}(-1)^{\ell (\sigma )}x^{L(\sigma )} &=
\left\{ \begin{array}{ll}
0, & \mbox{if $n \equiv 1 \pmod{2}$}, \\ 
x^{\frac{n}{2}(\frac{n}{2}-1)}, & \mbox{if $n \equiv 0 \pmod{2}$.}
\end{array} \right.  
\label{altplus} 
\end{align}
\end{pro}
\begin{proof}
Note that $E_n^{-} = \D_J^I(S_n)$ where $I := \{ i \in [n-1] : \, i \equiv 0 \pmod{2} \}$
and $J := \{ i \in [n-1] : \, i \equiv 1 \pmod{2} \}$ so the first equation in (\ref{altminus}) 
follows from Proposition~\ref{zero}. 
So assume that $n \equiv 0 \pmod{2}$, say $n=2m$ for some $m \in \N$.
By Lemma \ref{chdes} we have 
\[
\sum_{\sigma \in E_n^{-}} \f =\sum_{\sigma \in \D_J^I(C_n)} \f.
\]
We claim that $\D_J^I(C_{n,+})=\emptyset$. Let $\sigma \in \D_J^I(C_{n,+})$.
Let $i := \sigma^{-1}(1)$. Then $i \equiv 1 \pmod{2}$ so $i \in J$ and hence 
$\sigma(i) > \sigma(i+1)$ which is a contradiction.
Let now $\sigma \in \D_J^I(C_{n,-})$. We claim that then
\[
\sigma = [2 , 1 , 4 , 3 , 6 , 5, \dots, 2m , 2m-1].
\]
We prove this claim by induction on $m \in \mathbb N$. If $m=1$ the claim is clear. Let $m \geq 2$.
Let $ a := \sigma^{-1}(2m-1)$. Then $a \equiv 0 \pmod{2}$ so $a=2m$ (else 
$\sigma(a-1),\sigma(a+1)>\sigma(a)=2m-1$) and hence $\sigma(2m-1)=2m$. But 
$\sigma_{|[2m-2]} \in \D_{J \cap [n-3]}^{I \cap [n-3]}(C_{n-2,-})$ so the claim follows by induction.
Since $\ell([2 , 1 , 4 , 3, \dots, 2m , 2m-1])=m=L([2, 1 , 4 , 3, \dots 2m , 2m-1])$
the second equation in (\ref{altminus}) follows. 

Since the map $\sigma \mapsto w_0 \, \sigma$ is an involution between $E_n^+$ and $E_n^{-}$,
the equations in (\ref{altplus}) follow from those in (\ref{altminus}) and Proposition~\ref{baspro}.
\qedhere\end{proof}

We now consider a general family of descent classes which includes all quotients.
Let $I,J\subseteq [n-1]$. We say that $I$ and $J$ are {\em unmixed} if 
\begin{equation}
\label{hypdes}
I\cap J =(I+1)\cap J=I\cap (J+1)=\emptyset.
\end{equation}

Let $I,J \subseteq [n-1]$ be unmixed.
Let $I_1, \ldots, I_s$ be the connected components of $I$ and $J_1,\ldots, J_t$ be those of $J$.
We say that $(I,J)$ is \textit{compressed} if $|I_1|\equiv
\cdots\equiv |I_s|\equiv |J_1|\equiv \cdots \equiv |J_t|\equiv 1 \pmod
2$ and $|[n-1]\setminus (I\cup J)|=s+t-1$. For instance,
$(\{ 1,7,8,9 \}, \{ 3,4,5,11,12,13\})$ is compressed for $n=14$ while $(\{ 1,3 \}, \{ 7,8,9,11,12,13\})$ 
is not. Note that if $I,J \subseteq [n-1]$  are unmixed and  $(I,J)$ is compressed then $ n-1=|I|+|J|+s+t-1 \equiv 1 \pmod{2}$ so $n$ is even.

Let now $n =2m \in \mathbb N$ and let $I,J$ be unmixed with connected components $I_1, \ldots, I_s$, and $ J_1,\ldots, J_t$, respectively.
Then $I_1, \ldots, I_s, J_1,\ldots, J_t$ are the connected components of $I \cup J$. Therefore
$ \sum_{j=1}^{s}\left( \frac{|I_j|+1}{2} \right) + \sum_{k=1}^{t}\left( \frac{|J_k|+1}{2} \right) \leq m$, with equality holding if and only if
 $(I,J)$ is compressed.

We can now state one of the main results of this section.

\begin{thm}\label{mainchess}
Let $I,J\subseteq [n-1]$ be unmixed. Let $I_1, \ldots, I_s$ be the connected components of~$I$ and 
$J_1,\ldots, J_t$ be the connected components of $J$. Then
we have
\begin{equation}\label{odd}
\sum_{\sigma \in \D^{I}_{J}(C_{n,+})}{(-1)^{\ell(\sigma)}x^{L(\sigma)}}=
(-x)^{\dt}x^{\alpha(J)} \mcdesa {\displaystyle \prod_{k=b+d+1}^{m}(1-x^{2k})},
\end{equation}
if $n$ is odd, while
\begin{align}
\sum_{\sigma \in \D^{I}_{J}(C_{n,\,+})}(-1)^{\ell (\sigma )}x^{L(\sigma )} &=
\left\{ \begin{array}{ll}
(-x)^{\dt}x^{\alpha(J)}\frac{[\bt]_{x^2}}{[m]_{x^2}} \mcdesa, & \mbox{if $m=b+d$}, \\ 
(-x)^{\dt}x^{\alpha(J)} \mcdesa {\displaystyle \prod_{k=b+d+1}^{m-1} (1-x^{2k})}, & \mbox{otherwise,}
\end{array} \right.  
\label{evplus} 
\end{align}
and 

\begin{align}
\sum_{\sigma \in \D^{I}_{J}(C_{n,\,-})}(-1)^{\ell (\sigma )}x^{L(\sigma )} &=
\left\{ \begin{array}{ll}
(-1)^{\dt}x^{\bt+\alpha(J)}\frac{[\dt]_{x^2}}{[m]_{x^2}} \mcdesa, & \mbox{if $m=b+d$}, \\ 
-(-x)^{\dt}x^{m+\alpha(J)} \mcdesa {\displaystyle \prod_{k=b+d +1}^{m-1} (1-x^{2k})}, & \mbox{otherwise,}
\end{array} \right.  
\label{evminus} 
\end{align}
if $n$ is even, where $m:=\left\lfloor \frac{n}{2}\right\rfloor$,
$b_j:= \left\lfloor\frac{|I_j|+1}{2}\right\rfloor$, for $j=1,\ldots, s$, $d_k:=\left\lfloor\frac{|J_k|+1}{2}\right\rfloor$, for $k=1,\ldots, t$, $\bt:=\sum_{i=1}^{s} b_i$, $\dt:=\sum_{k=1}^{t} d_k$,
${\bf b}:=b_1,\ldots,b_s$, ${\bf d}:=d_1,\ldots, d_t$, and $\alpha(J):=\sum_{k=1}^t d_k^{2}$. 
\end{thm}
\begin{proof}
We let, for convenience, $\bar{b}_j:=b_j +1, \;\;  \bar{d}_k:=d_k+1$, for $j \in [s]$ and $k \in [t]$,
$\hat\alpha(J):=\alpha(J)+\dt$, and $\check\alpha(J):=\alpha(J)-\dt$.

Before delving into the proof we think it useful to sketch the idea of it. If $J$ has at least one 
connected component of even size then by Lemma \ref{destoas} this can be changed to a connected component 
of $I$ and we can proceed by induction. If $I$ has a connected component of even size (and all 
connected components of $J$ have odd size) then by Propositions \ref{scr} and \ref{scl} we can
remove one of the endpoints from this connected component and then shift some of the other connected 
components of $I$ and $J$ so that the resulting ``empty spot'' sits next to a connected component of
$J$, to which it can then be ``added'' by Propositions \ref{shrdes} or \ref{shldes}. The resulting
descent class now has a connected component of the descents of even size so can be computed by induction as 
in the previous case. If all the connected components of $I \cup J$ are of odd size but $(I,J)$ is not
compressed then there is either an ``empty spot'' to the right of the rightmost connected component of 
$I \cup J$, or to the left of the leftmost, or there are two consecutive connected components of 
$I \cup J$ separated by at least two empty spots. By shifting the connected components of $I \cup J$
we can ``move'' this extra empty spot so that it sits next to a connected component of $J$, to which it 
can then be ``added'', and we can conclude as in the previous case. If $(I,J)$ is compressed then $n$ is
even and must appear immediately to the right of a connected component of the ascents which allows us to
``delete'' $n$ and compute the generating function as a sum of generating functions of unmixed descent 
classes of $S_{n-1}$.

We proceed by induction on $t\in \N_0$, the number of connected components of the descents. 
Let $t=0$ (i.e., $J=\emptyset$). Then $ (I, \emptyset)$ is compressed if and only if $ \bt = \frac{n}{2}$
 so Theorem \ref{main} reduces to Theorem~\ref{thmA} in this case.
Let now $t \geq 1$. 

Assume first that there exists $i\in [t]$ such that $|J_i|\equiv 0 \pmod 2$. 
Then by Lemma~\ref{destoas} and our induction hypothesis we have
\begin{align*}
&\sum_{\sigma \in \D^{I}_{J}(C_{n,+})}{\!\!(-1)^{\ell(\sigma)}x^{L(\sigma)}}\!\!=\!(-1)^{\frac{|J_i|}{2}} x^{\frac{|J_i|(|J_i|+2)}{4}} \sum\limits_{\sigma \in \D^{I\cup{J_i}}_{J\setminus J_i}(C_{n,+})}{(-1)^{\ell(\sigma)}x^{L(\sigma)}}\\
&= (-1)^{d_{i}} x^{d_i \bar{d}_i} (-1)^{\dt-d_i} x^{\hat\alpha(J)- d_i \bar d _i} \mcdes \prod\limits_{k=\bt+\dt+1}^{\left\lfloor\frac{n-1}{2}\right\rfloor}(1-x^{2k}),
\end{align*}
so \eqref{odd} and the second formula in \eqref{evplus} follow in this case.

Under the same hypothesis, for the odd chessboard elements we have
\begin{align*}
&\sum_{\sigma \in \D^{I}_{J}(C_{2m,-})}{(-1)^{\ell(\sigma)}x^{L(\sigma)}}=(-1)^{\frac{|J_i|}{2}} x^{\frac{|J_i|(|J_i|+2)}{4}} \sum\limits_{\sigma \in \D^{I\cup{J_i}}_{J\setminus J_i}(C_{2m,-})}{(-1)^{\ell(\sigma)}x^{L(\sigma)}}\\
&= -(-1)^{d_{i}} x^{d_i \bar{d}_i} (-1)^{\dt - d_i} x^{m+\hat\alpha(J)-d_i\bar d _i} \mcdes \prod\limits_{k=\bt+\dt+1}^{m-1}(1-x^{2k}),
\end{align*}
yielding the second formula in~\eqref{evminus}.

We may therefore assume that $|J_1|\equiv|J_2|\equiv\cdots\equiv|J_t|\equiv 1\pmod 2$.

Assume now that there exists $r\in [s]$ such that $|I_r|\equiv 0 \pmod 2$. Then by repeated application of Proposition~\ref{shrdes} and \ref{shldes}, we have 
\[\sum_{\sigma \in \D^{I}_{J}(C_n,\,\pm)}{(-1)^{\ell(\sigma)}x^{L(\sigma)}}=\sum_{\sigma \in \D^{\widetilde{I}}_{\widetilde{J}}(C_n,\,\pm)}{(-1)^{\ell(\sigma)}x^{L(\sigma)}},\]
where $\widetilde{I}$ has connected components $\widetilde{I}_1\cup\cdots\cup\widetilde{I}_s,$ where $|\widetilde{I}_r|=|I_r|-1$ and $|\widetilde{I}_k|=|I_k|$, for $k\in [s]\setminus\{r\}$ and $\widetilde{J}$ has connected components $\widetilde{J}_1\cup\cdots\cup\widetilde{J}_t,$ where $ |\widetilde{J}_1|=|J_1|+1$  and $ |\widetilde{J}_k|=|J_k|,$ for $k\in [2,t],$ and the connected components of $\widetilde{I}\cup\widetilde{J}$ are $\widetilde{I}_1, \ldots, \widetilde{I}_s, \widetilde{J}_1,\ldots ,\widetilde{J}_t$.
Since $\widetilde{J}$ has a connected component of even cardinality, reasoning as in the previous case, 
and observing that    $\left\lfloor\frac{|\widetilde{J}_1|+1}{2}\right\rfloor=
\left\lfloor\frac{|J_1|+1}{2}\right\rfloor=d_1$ and $\left\lfloor\frac{|\widetilde{I}_r|+1}{2}
\right\rfloor=\left\lfloor\frac{|I_r|+1}{2}\right\rfloor=b_r$, we conclude again by induction.

We may therefore assume that  $|I_1|\equiv\cdots\equiv|I_s|\equiv|J_1|\equiv\cdots\equiv|J_t|\equiv 1\pmod 2.$

Suppose first that $|[n-1]\setminus(I\cup J)|>s+t-1$. Therefore either $1\notin I \cup J$ or $n-1 \notin I\cup J$ or there exists $i\in [n-1]$ such that $i,\,i+1 \notin I \cup J.$ In any of these cases we can apply Propositions~\ref{shrdes} and \ref{shldes} to get
\[\sum_{\sigma \in \D^{I}_{J}(C_n,\,\pm)}{(-1)^{\ell(\sigma)}x^{L(\sigma)}}=\sum_{\sigma \in \D^{\bar{I}}_{\bar{J}}(C_n,\,\pm)}{(-1)^{\ell(\sigma)}x^{L(\sigma)}},\]
where $\bar I$ has connected components $\bar{I}_1,\ldots ,\bar{I}_s$ such that 
$|\bar{I}_j|=|I_j|$ for $j\in [s]$ and $\bar{J}$ has connected components $\bar{J}_1,\ldots ,\bar{J}_t$ 
such that $|\bar{J}_1|=|J_1|+1$ and $|\bar{J}_l|=|J_l|$ for $l \in [2,t]$. Then, again, $\bar{J}$, 
has a connected component of even size so, reasoning as above
 \eqref{odd}, and the second equations in \eqref{evplus} and \eqref{evminus} follow by induction, since
$\left\lfloor\frac{|\bar{J}_1|+1}{2}\right\rfloor=\left\lfloor\frac{|J_1|+1}{2}\right\rfloor =d_1$.

\vspace{3mm}

We may therefore assume that $|I_1|\equiv\cdots\equiv|I_s|\equiv|J_1|\equiv\cdots\equiv|J_t|\equiv 1\pmod 2$ and $|[n-1]\setminus(I\cup J)|=s+t-1$, i.e., that $(I,J)$ is compressed. 
Then $n\equiv 0 \pmod 2$, say $n=2m$, and $m=\bt+\dt$, and both the leftmost and the rightmost elements 
of any connected component of $J \cup J$ are odd.

For $i\in [s]$ let $a_i := \max{I_i} +1$ and for $i \in [t]$ let $c_i := \min{J_i}$. Then $a_1\equiv \cdots \equiv a_s \equiv 0 \pmod 2$ and $c_1\equiv \cdots \equiv c_t \equiv 1 \pmod 2$. Therefore, if $\sigma \in \D_J ^I (C_{2m,+})$, then $\sigma^{-1}(2m)\in\{a_1,\ldots,a_s\}$. Hence
\[
\sum_{\sigma \in \D^{I}_{J}(C_{2m,\,+})}{(-1)^{\ell(\sigma)}x^{L(\sigma)}}=
\sum_{j=1}^{s}\,{\sum_{\substack{\sigma \in \D^{I}_{J} (C_{2m,\,+}): \\  \sm (2m)=a_j }} {(-1)^{\ell(\sigma)}x^{L(\sigma)}}}.\]
Fix $j\in [s]$.
Let $k:= \max\{i \in [t]: c_i<a_j\}$ (where $k:=0$ if $\{i \in [t]: c_i<a_j\}=\emptyset$).
So $J_1, \ldots , J_k$ are to the left of $a_j$, while $J_{k+1}, \ldots , J_t$ are to the right. 
Let $\bar{\tau}$ be obtained from $\tau$ by removing the maximum (which is in position $a_j$) and 
reversing the elements in each of the  blocks of ascents and descents that are to the right of $a_j$,
so reversing the elements in positions $[min I_i,a_i]$ for each $i=j+1, \ldots , s$, 
and those in positions $[c_i, \max J_i+1]$ for each $i=k+1, \ldots , t$. Then the map 
$\tau \mapsto \bar{\tau}$ is a bijection between $\{ \sigma \in \D^{I}_{J}(C_{2m,\,+})\,:\, \sm(2m)=a_j\}$ and 
$\D^{\varphi_j(I)}_{\varphi_j (J)} (C_{2m-1})$,
where $\varphi_{j}(I) :=I_{1}\cup \cdots \cup I_{j-1} \cup (I_{j} \setminus \{a_{j} -1\}) \cup (J_{k+1} -1) \cup \cdots \cup (J_{t}-1) $
and $\varphi_{j}(J) :=J_{1}\cup \cdots \cup J_{k} \cup (I_{j+1} -1) \cup \cdots \cup (I_{s}-1) .$

Furthermore, we have  $\ell(\bar{\tau})=\ell(\tau)+A$ and $L(\bar{\tau})=L(\tau)+B$, where,
by Proposition~\ref{baspro}
\small
\begin{align*}
A=& \sum_{r=j+1}^{s}{{{|I_r|+1} \choose 2}}-\sum_{h=k+1}^{t}{{{|J_h|+1} \choose 2}} -(2m-a_j)=\!\sum_{r=j+1}^{s}{b_r(2b_r-1)}-\!\!\sum_{h=k+1}^{t}{d_h(2d_h-1)} -(2m-a_j)\\
& =\sum_{r=j+1}^{s}{b_r(2b_r-3)}-\sum_{h=k+1}^{t}{d_h(2d_h+1)} , \\
B=&\sum_{r=j+1}^{s}{\left({\frac{|I_r|+1}{2}}\right)^2}-
\sum_{h=k+1}^{t}{\left(\frac{|J_{h}|+1}{2}\right)^2} -\frac{2m-a_j}{2} =\sum_{r=j+1}^{s}{b_r (b_r-1)}-\sum_{h=k+1}^{t}{d_h(d_h+1)},
\end{align*}
\normalsize
since $2m-a_j=2\left(\sum_{r=j+1}^s b_r +\sum_{h=k+1}^t d_h\right)$.
Therefore, by our induction hypothesis~\eqref{odd}, 
\begin{eqnarray}\nonumber
\sum_{\substack{ \tau \in \D^{I}_{J} (C_{2m,\,+}): \\  \sm (2m)=a_j }} {(-1)^{\ell(\tau)}x^{L(\tau)}}\!\! &=&\!\!(-1)^A x^{-B} \sum_{\bar\tau \in \D^{\varphi_j(I)}_{\varphi_j (J)} (C_{2m-1})}{(-1)^{\ell(\bar{\tau})}x^{L(\bar{\tau})}} \\ \label{ccc}
&=& (-1)^{\dt} x^{\hat{\alpha}(\varphi_j(J))-B}{\footnotesize
    \left[ \begin{array}{c} m-1\\
b_1,\ldots,b_{j-1},b_j -1, b_{j+1},\ldots,b_s,{\bf d} \end{array} \right]_{x^2}}.\end{eqnarray}
But $
\hat\alpha(\varphi_j(J))= \sum\limits_{r=1}^k{d_r(d_r +1)} +\sum\limits_{r=j+1}^s{b_r(b_r+1)}$, so $\hat\alpha(\varphi_j(J))-B=\hat\alpha(J) + 2\sum\limits_{r=j+1}^s b_r$.

Thus, the sum in \eqref{ccc} becomes
\[
(-1)^{\dt} x^{\hat\alpha(J)} x^{2\sum_{r=j+1}^{s} b_{r}}{\footnotesize \left[ \begin{array}{c} m-1\\
b_1,\ldots,b_{j-1}, b_j -1, b_{j+1},\ldots,b_s,{\bf d} \end{array} \right]_{x^2}}.
\]
Therefore
\begin{eqnarray*}
\sum_{\sigma \in \D^{I}_{J}(C_{2m,+})}{\!\!\!\!(-1)^{\ell(\sigma)} x^{L(\sigma)}}\! \!&=&\!\! (-1)^{\dt} x^{\hat\alpha(J)} \sum_{j=1}^{s}{{x^2}^{\sum_{r=j+1}^{s}{b_r}}}{\footnotesize\left[ \begin{array}{c} m-1\\
b_1,\ldots,b_{j-1},b_j -1, b_{j+1},\ldots,b_s,{\bf d}  \end{array} \right]_{x^2}} \\
&=&\!\! (-1)^{\dt} x^{\hat\alpha(J)} \frac{[\bt ]_{x^2}}{[m]_{x^{2}}} \footnotesize{ \left[ \begin{array}{c} m \\
\bf b, \,\bf d \end{array} \right]_{x^2}} 
\end{eqnarray*}
as desired.

Under the same hypothesis, for the sum over odd chessboard elements we have, by Proposition~\ref{destoas2}
and Remark~\ref{rmk:res} 
\begin{eqnarray*}
\sum_{\sigma \D^{I}_{J}(C_{2m,-})}{(-1)^{\ell(\sigma)}x^{L(\sigma)}}&=& (-1)^{\ell(w_0)}x^{L(w_0)}  \sum_{\tau \in \D^{J}_{I}(C_{2m,+})}{(-1)^{\ell(\tau)} x^{-L(\tau)}} \\
\\
&=& (-1)^{2m \choose 2} x^{m^{2}} \sum_{\tau \in \D^{J}_{I}(C_{2m,+})}{(-1)^{\ell(\tau)} x^{-L(\tau)}} \\
\\
&=& (-1)^m x^{m^2} (-1)^{\bt} x^{-\sum_{j=1}^{s} {b_j \bar{b}_j}} \frac{[d]_{x^{-2}}}{[m]_{x^{-2}}} {\footnotesize{ \left[ \begin{array}{c} m \\
\bf b , \bf d \end{array} \right]}_{x^{-2}}} \\
\\
&=& (-1)^{\dt}x^{m+\check\alpha(J)}\frac{[\dt]_{x^2}}{[m]_{x^2}} \footnotesize{ \left[ \begin{array}{c} m \\
\bf b, \bf d \end{array} \right]_{x^2}}
\end{eqnarray*}
and the result follows. 
This concludes the proof of the first equations in \eqref{evplus} and \eqref{evminus} and hence of the result. \qedhere
\end{proof}

By Lemma~\ref{chdes} the preceding result implies the following one, which computes the sign-twisted generating
function of the odd length over any unmixed descent class.

\begin{thm}
\label{main}
Let $I,J\subseteq [n-1]$ be unmixed. 
Then, keeping the same notation as in Theorem~\ref{mainchess}
\[
\sum_{\sigma \in \D^{I}_{J}(S_n)}(-1)^{\ell (\sigma )}x^{L(\sigma )} =\begin{cases}
(-1)^{\dt} x^{\alpha(J)} \dfrac{x^{\dt} [\bt]_{x^2} + x^{\bt}[\dt]_{x^2}}{[b+d]_{x^2}} 
\footnotesize{ \left[ \begin{array}{c} b+d \\ \bf b, \,\bf d \end{array} \right]_{x^2}}, & \mbox{if $n=2(b+d)$}, \\  &\\
(-x)^{\dt}x^{\alpha(J)} 
\footnotesize{ \left[ \begin{array}{c} b+d \\ \bf b, \,\bf d \end{array} \right]_{x^2}} 
{\displaystyle \prod_{k=2b+2d+2}^{n} (1+(-1)^{k-1} x^{\lfloor \frac{k}{2} \rfloor})}, & \mbox{otherwise.}
\end{cases}   
\]
\end{thm}

\vspace{1em}

\section{Open problems}\label{sec:conj}

In this section we collect some conjectures and open problems
arising from this work.

For $\sigma\in S_n$, we let 
 $\cl_n(\sigma)=\{\tau\in S_n : D_o(\tau)=D_o(\sigma)\}$
denote the equivalence  class of permutations in $S_n$ with the same  odd
diagram as $\sigma$.
Clearly, the  problem of characterizing the odd diagrams is closely related to that of identifying these equivalence classes.

Recall that a permutation $\sigma\in S_n$ is said to \emph{contain the
  pattern}~$\alpha=\alpha_1\cdots \alpha_k$  if there exist $1\leq
i_1<\dots<i_k\leq n$ such that  $\sigma(i_1),\dots,\sigma(i_k)$ are in
the same relative order as  $\alpha_1,\dots,\alpha_k$. 
A permutation $\sigma\in S_n$ is said to \emph{avoid} the pattern~$\alpha$
if it does not contain the pattern $\alpha$.  
We denote with $\Av_n(\alpha)=\{\sigma\in S_n : \sigma \mbox{ avoids
}\alpha\}$  the set of permutations of degree $n$ avoiding~$\alpha$. 
We conjecture that odd diagrams faithfully encode permutations
avoiding some patterns  of length $3$. More precisely,   we conjecture the following.
\begin{con}\label{thm:avoid}
 Let $\alpha\in\{213,312\}$. The map $D_o\colon \Av_n(\alpha)\to
 [n]^2,\, \sigma\mapsto D_o(\sigma)$  is injective. More precisely, for a
 permutation $\sigma \in S_n$,   the class $\cl_n(\sigma)$ contains at
 most  one permutation avoiding the pattern $213$ and at most one avoiding
 $312$.  If they exist, they are respectively the longest and the
 shortest  element of $\cl_n(\sigma)$.
\end{con}

We have verified Conjecture \ref{thm:avoid} for $n \leq 7$.
\vspace{5mm}

In light of Proposition~\ref{pro:einv} and Remark~\ref{rmk:modinv},
it  is natural to investigate the polynomials giving the (non-twisted)
distribution of  the odd inversions. For $n\in \mathbb N$ we denote
this  polynomial by $L_n(x):=\sum_{\sigma\in S_n}x^{L(\sigma)}$.
Properties (iii) and (iv) in Proposition~\ref{baspro} imply that
$L_n(x)$  is monic and symmetric for all $n\in \mathbb{N}$. For small values of $n$ we have: 
\begin{eqnarray*}
L_3(x)&=&1+4x+x^2 \\
L_4(x)&=&1+8x+6x^2+8x^3+x^4 \\
L_5(x)&=&1+12x+23x^2+48x^3+23x^4+12x^5+x^6\\
L_6(x)&=&1+16x+59x^2+137x^3+147x^4+147x^5+137x^6+59x^7+16x^8+x^9
\end{eqnarray*}
With the exception of $n=4$, for $n\leq 11$ the polynomials $L_n(x)$ are unimodal. We therefore conjecture the following.
\begin{con}\label{con:oddunim}
	Let $n\geq 5$. Then the polynomial $L_n(x)$ is unimodal.
\end{con}
The first rows of the associated triangle are recorded in~\cite[A289511]{Oeis}.

\vspace{5mm}

Generalizing further from Remark~\ref{rmk:modinv}, we let, for $k,n\in \mathbb N$, $h\in \mathbb Z/ k\mathbb Z$ and $\sigma\in S_n$,
    \begin{equation}\label{eq:hmod}
      \inv_{k,h}(\sigma)=|\{(i,j)\in [n]^2 : i<j, \sigma(i)>\sigma(j), j-i\equiv h \pmod k\}|.
    \end{equation}
    Note that $L(\sigma)=\inv_{2,1}(\sigma)$. Also, note that for
    $k\geq n-1$, the polynomials of the distributions of the statistic $\inv_{k,1}$ over $S_n$ coincide with the Eulerian polynomials:
    $$\sum_{\sigma\in S_n} x^{\inv_{k,1}(\sigma)}=\sum_{\sigma\in S_n} x^{\des(\sigma)},$$ where $\des(\sigma)=|\Des(\sigma)|$ denotes the descent number of the permutation $\sigma$.
    Inspired by this fact and  Conjecture~\ref{con:oddunim}, we put forward the following  general conjecture.
    \begin{con}
    \label{1modk}
      The polynomials
      $$\sum_{\sigma\in S_n} x^{\inv_{k,1}(\sigma)}$$
      are unimodal for all $n\in \mathbb N$, and all $k\geq 3$.
      \end{con}
We have verified that Conjecture \ref{1modk}  holds for $n \leq 9$, and all relevant $k$.

\vspace{4mm}

We have seen in Proposition~\ref{zero} some sufficient conditions for the sign-twisted generating function
of the odd length to be zero on a descent class. This, together with the comments following Proposition~\ref{zero}, suggests
 the following natural problem.
\begin{prob}
Let $I,J \subseteq [n]$, $I \cap J = \emptyset$. Give necessary and sufficient conditions on $I$ and $J$ 
such that
\[
\sum_{\sigma \in \D^{I}_{J}(S_n)}(-1)^{\ell (\sigma )}x^{L(\sigma )} =0.
\]
\end{prob}

\begin{acknow} The first author would like to thank 
A. Postnikov, A.\ Rapagnetta, R. Stanley, M. Wachs, and L.\  Williams 
for interesting and useful conversations. 
FB was partially supported by the MIUR Excellence Department Project 
CUP E83C18000100006.
The second named author was partially supported by the German-Israeli 
Foundation for Scientific Research and Development through grant no.~1246.  
AC would also like to thank the Erwin Schr\"odinger International Institute 
for Mathematics and Physics (Vienna), where part of this research was carried 
out.\end{acknow}


\begin{thebibliography}{xx}


\bibitem{BB}
A. Bj\"{o}rner, F. Brenti, {\em Combinatorics of Coxeter groups},
Graduate Texts in Mathematics, {\bf 231}, Springer-Verlag, New York, 2005.

\bibitem{BC}
F.\ Brenti, A.\ Carnevale, {\em Proof of a conjecture of Klopsch-Voll
  on Weyl groups of type $A$},   Trans. Amer. Math. Soc. {\bf 369} (2017), 7531--7547. 

\bibitem{BC2}
F.\ Brenti, A.\ Carnevale, {\em Odd length for even hyperoctahedral
  groups and signed generating functions},  Discrete Math.,  {\bf 340} (2017), 2822--2833.

\bibitem{BC4}
F.\ Brenti, A. Carnevale, {\em Odd length in Weyl groups}, Algebraic Comb., 
{\bf 2} (2019), no. 6, 1125--1147.

\bibitem{GH}
P. Griffiths, J. Harris, {\em Principles of algebraic geometry}, 
Pure and Applied Mathematics,
Wiley-Interscience, New York,  1978.

\bibitem{Har}
R. Hartshorne, {\em Algebraic geometry},
Graduate Texts in Mathematics, {\bf 52}, Springer-Verlag, New York-Heidelberg, 1977.

\bibitem{KV}
B. Klopsch, C. Voll, {\em Igusa-type functions associated to finite
  formed spaces and their functional equations},
Trans. Amer. Math. Soc.,  {\bf 361} (2009), no. 8, 4405--4436.

\bibitem{Turan}
    T. K\"{o}vari, V. T. S\'{o}s,P. Tur\'{a}n, {\em On a problem of
{K}.\ {Z}arankiewicz},  {Colloquium Math.},  {\bf 3 }, (1954), 50--57.   
            
\bibitem{Mac}
 I.G. Macdonald,
{\em Notes on Schubert polynomials}, Publications du LaCIM, Montreal, 1991. 

\bibitem{Oeis}
  N. J. A. Sloane, editor, {\em The On-Line Encyclopedia of Integer
 Sequences},  published electronically at \href{https://oeis.org}{https://oeis.org}, 2019.


\bibitem{StaEC2}
R. P. Stanley, {\em Enumerative Combinatorics}, vol.2, 
Cambridge Studies in Advanced Mathematics, no.62,
Cambridge Univ. Press, Cambridge,  1999.

\bibitem{SV1}
A. Stasinski, C. Voll, {\em A new statistic on the hyperoctahedral
  groups},  Electronic J. Combin., {\bf 20} (2013), no. 3, Paper 50, 23 pp. 

\bibitem{SV2}
A. Stasinski, C. Voll, {\em  Representation zeta functions of
  nilpotent groups  and generating functions for Weyl groups of type B}, Amer. J. Math., {\bf 136} (2) (2014), 501--550.

\bibitem{Stem}
  J. Stembridge,
{\em Sign-twisted Poincaré series and odd inversions in Weyl groups}, 
Algebraic Comb., {\bf 2} (2019), no. 4, 621--644.
  
\end{thebibliography}
\end{document}